\documentclass[twosided]{amsart}
  \usepackage{amsmath,amssymb,amsthm,amsfonts}
  \usepackage{hyperref}
  \usepackage{mathtools}
  \usepackage{latexsym}
  \usepackage{mathrsfs}
  \usepackage{pstricks}
   \usepackage{tikz-cd} 
  \usepackage{verbatim}
  \usepackage{float}
  \usepackage{booktabs}
  \usepackage{enumitem}
  \usepackage{textcomp}
  \usepackage{tikz}
 \usetikzlibrary{matrix,arrows,decorations.pathmorphing}
\usepackage{listings}

  \newcommand{\wt}{\widetilde}
  \newcommand{\gdeg}{G\text{\rm -deg}}

  \newcommand{\vp}{\varphi}
  \newcommand{\ve}{\varepsilon}

  \DeclareMathOperator{\id}{Id}
  \def\bn{\mathbb N}
\def\br{\mathbb R}
\def\bz{\mathbb Z}

\def\noi{\noindent}
  
  \newcommand{\vs}{\vskip .3cm}
  
  \newcommand{\amal}[5]{#1\prescript{#2}{}\times_{#3}^{#4}#5}

  \newcommand{\norm}[1]{\left\lVert#1\right\rVert}

  \newcommand\cV{\ensuremath{\mathcal V}}
  
  \newcommand\cW{\ensuremath{\mathcal W}}

  \newcommand\bbR{\ensuremath{\mathbb R}}

  \newcommand\bbV{\ensuremath{\mathbb V}}

  \newcommand\bbZ{\ensuremath{\mathbb Z}}

  \newcommand\bfV{\ensuremath{\mathbf V}}

  \newcommand\scrA{\ensuremath{\mathscr A}}

  \newcommand\scrE{\ensuremath{\mathscr E}}
  \newcommand\scrF{\ensuremath{\mathscr F}}

  \newrgbcolor{violet}{.6 .1 .8}
  \newrgbcolor{lightyellow}{1 1 .8}
  \newrgbcolor{lightblue}{.80 1 1}
  \newrgbcolor{mygreen}{0 .66 .05}
  \definecolor{mygreen}{rgb}{0,.66,.05}
  \definecolor{lightyellow}{rgb}{1,1,.80}
  \newrgbcolor{orange}{1 .6 0}
  \newrgbcolor{GreenYellow}{.85 1 .31}
  \newrgbcolor{Yellow}{1  1  0}
  \newrgbcolor{Goldenrod}{1  .90  .16}
  \newrgbcolor{Dandelion}{1  .71  .16}
  \newrgbcolor{Apricot}{1  .68  .48}
  \newrgbcolor{Peach}{1  .50  .30}
  \newrgbcolor{Melon}{1  .54  .50}
  \newrgbcolor{YellowOrange}{1  .58  0}
  \newrgbcolor{Orange}{1  .39  .13}
  \newrgbcolor{BurntOrange}{1  .49  0}
  \newrgbcolor{Bittersweet}{1.  .4300  .24}
  \newrgbcolor{RedOrange}{1  .23  .13}
  \newrgbcolor{Mahogany}{1.  .4475  .4345}
  \newrgbcolor{Maroon}{1.  .4084  .5376}
  \newrgbcolor{BrickRed}{1.  .3592  .3232}
  \newrgbcolor{Red}{1  0  0}
  \newrgbcolor{OrangeRed}{1  0  .50}
  \newrgbcolor{RubineRed}{1  0  .87}
  \newrgbcolor{WildStrawberry}{1  .04  .61}
  \newrgbcolor{CarnationPink}{1  .37  1}
  \newrgbcolor{Salmon}{1  .47  .62}
  \newrgbcolor{Magenta}{1  0  1}
  \newrgbcolor{VioletRed}{1  .19  1}
  \newrgbcolor{Rhodamine}{1  .18  1}
  \newrgbcolor{Mulberry}{.6668  .1180  1.}
  \newrgbcolor{RedViolet}{.9538  .4060  1.}
  \newrgbcolor{Fuchsia}{.5676  .1628  1.}
  \newrgbcolor{Lavender}{1  .52  1}
  \newrgbcolor{Thistle}{.88  .41  1}
  \newrgbcolor{Orchid}{.68  .36  1}
  \newrgbcolor{DarkOrchid}{.60  .20  .80}
  \newrgbcolor{Purple}{.55  .14  1}
  \newrgbcolor{Plum}{.50  0  1}
  \newrgbcolor{Violet}{.98 .15 .95}
  \newrgbcolor{RoyalPurple}{.25  .10  1}
  \newrgbcolor{BlueViolet}{.84  .38  .98}
  \newrgbcolor{Periwinkle}{.43  .45  1}
  \newrgbcolor{CadetBlue}{.38  .43  .77}
  \newrgbcolor{CornflowerBlue}{.35  .87  1}
  \newrgbcolor{MidnightBlue}{.4414  .9259  1.}
  \newrgbcolor{NavyBlue}{.06  .46  1}
  \newrgbcolor{RoyalBlue}{0  .50  1}
  \newrgbcolor{Blue}{0  0  1}
  \newrgbcolor{Cerulean}{.06  .89  1}
  \newrgbcolor{Cyan}{0  1  1}
  \newrgbcolor{ProcessBlue}{.04  1  1}
  \newrgbcolor{SkyBlue}{.38  1  .88}
  \newrgbcolor{Turquoise}{.15  1  .80}
  \newrgbcolor{TealBlue}{.1572  1.  .6668}
  \newrgbcolor{Aquamarine}{.18  1  .70}
  \newrgbcolor{BlueGreen}{.15  1  .67}
  \newrgbcolor{Emerald}{0  1  .50}
  \newrgbcolor{JungleGreen}{.01  1  .48}
  \newrgbcolor{SeaGreen}{.31  1  .50}
  \newrgbcolor{Green}{0  1  0}
  \newrgbcolor{ForestGreen}{.1992  1.  .2256}
  \newrgbcolor{PineGreen}{.3100  1.  .5575}
  \newrgbcolor{LimeGreen}{.50  1  0}
  \newrgbcolor{YellowGreen}{.56  1  .26}
  \newrgbcolor{SpringGreen}{.74  1  .24}
  \newrgbcolor{OliveGreen}{.6160  1.  .4300}
  \newrgbcolor{RawSienna}{.53  .28  .16}
  \newrgbcolor{Sepia}{1.  .7510  .70}
  \newrgbcolor{Brown}{.41  .25  .18}
  \newrgbcolor{TAN}{.86  .58  .44}
  \newrgbcolor{Gray}{1.  1.  1.}
  \newrgbcolor{Black}{1  1  1}
  \newrgbcolor{White}{1  1  1}

  \theoremstyle{plain}
  \newtheorem{theorem}{Theorem}[section]
  \newtheorem{proposition}[theorem]{Proposition}
  \newtheorem{lemma}[theorem]{Lemma}
  \newtheorem{corollary}[theorem]{Corollary}

  \theoremstyle{definition}
  
  \newtheorem{definition}[theorem]{Definition}
  \newtheorem{remark}[theorem]{Remark}

\title[Periodic Solutions to Reversible 2nd Order DDEs in Prescribed Domains]
  {Periodic Solutions to Reversible Second Order Autonomous DDEs in Prescribed Symmetric Nonconvex Domains}

\author[Z. Balanov]{Zalman Balanov}

\address{Department of Mathematics, Xiangnan University, 889 Chen Zhou Da Dao, Chenzhou, Hunan 423000, China, and  
 Department of Mathematical Sciences, the University of Texas at Dallas, Richardson, Texas 75080, USA.}

 \email{balanov@utdallas.edu}
 
 \author[N. Hirano]{Norimichi Hirano}

 \address{Graduate School of Environment and Information Sciences, Yokohama National University, Yokohama, Japan}   
 
 \email{hira0918@gmail.com}
 
\author[W. Krawcewicz]{Wies{\l}aw Krawcewicz}
\address{Applied Mathematics Center at Guangzhou University,
    Guangzhou 510006, China, and
   Department of Mathematical Sciences, the University of Texas at Dallas,
   Richardson, Texas 75080, USA.}
   
   \email{wieslaw@utdallas.edu}

\author[F. Liao]{Fangfang Liao$^1$}\thanks{$^1$ Fangfang Liao is the corresponding author}

\address{Department of Mathematics, Xiangnan University, 889 Chen Zhou Da Dao, Chenzhou, Hunan 423000, China}

\email{liaofangfang1981@126.com}

\author[A. Murza]{Adrian Murza}

\address{Department of Mathematical Sciences, the University of Texas at Dallas, Richardson, Texas 75080, USA.}

\email{Adrian.Murza@utdallas.edu}

\begin{document}
  	
\maketitle
  
 \begin{abstract}
The existence and spatio-temporal patterns of $2\pi$-periodic solutions  
to second order reversible equivariant autonomous systems with commensurate delays 
are studied 
using the Brouwer $O(2) \times \Gamma \times \mathbb Z_2$-equivariant degree theory. The solutions are supposed to take their values in a prescribed symmetric domain $D$, while $O(2)$ is related to the reversal symmetry combined with 
the autonomous form of the system. The group  $\Gamma$ reflects symmetries of $D$ and/or possible coupling  in the corresponding network of identical oscillaltors, and $\mathbb Z_2$ is related to the oddness of the right-hand side. Abstract results, based on the use of Gauss curvature of $\partial D$, Hartman-Nagumo type {\it a priori bounds} and Brouwer equivariant degree  techniques, are  supported by a concrete example with $\Gamma = D_8$ -- the dihedral  group of  order $16$. 
  \end{abstract}


  	
  \section{Introduction}\label{sec:introduction}
  
\

{\bf (a) Subject and goal.}   Existence of periodic solutions to equivariant dynamical systems together with  describing their spatio-temporal symmetries constitute an important problem of equivariant dynamics (see, for example, \cite{GolSchSt,GolStew} for the equivariant singularity theory based methods and \cite{AED,survey,IV-book} for the equivariant degree treatment). As is well-known, second order systems of ODEs with no friction term exhibit an extra symmetry -- the so-called reversal symmetry, i.e. if $x(t)$ is a solution to the system, then so is $x(-t)$. We refer to \cite{Lamb-Roberts} for a comprehensive exposition of (equivariant) reversible ODEs as well as their applications in natural sciences (see also \cite{AS}). It should be stressed that in the context relevant to spatio-temporal symmetries of periodic solutions, the reversal symmetry gives rise to extra subgroups of the non-abelian group $O(2)$.

Simple examples show that, in contrast to their ODEs counterparts, second order delay differential equations (in short, DDEs) with no friction term are not reversible, in general. In \cite{BW} (see also \cite{KW}), we considered  {\it space reversible} equivariant mixed DDEs of the form
\begin{align}
\label{eq:dde_system}
\ddot v(y)=g(\alpha, v(y))+a(v(y-\alpha)+v(y+\alpha)), \quad a,\alpha \in\br ,
\end{align}
 with equivariant $g : \mathbb R^n \to \mathbb R^n$ 
(one can think of equations governing steady-state
solutions to PDEs, cf.~\cite{Lamb-Roberts} and references therein). Note that by replacing $y$ by $t$ in \eqref{eq:dde_system}, 
one obtains {\it  time-reversible} DDEs. However, such systems involve using the information from the future by ``traveling back in time'', which is 
difficult to justify from a commonsensical viewpoint.   
  
Time delay systems with {\it commensurate delays} play an important role in robust control theory (see, for example, \cite{GuoKharitonovCheng} and references therein). A class of such systems exhibiting a reversal symmetry is the main {\it subject} of the present paper. To be more specific, we are interested in the periodic problem
\begin{align}\label{eq:1}
\left\{  \begin{aligned}
 \ddot{x}(t)&=f\left(x(t),x\left(t-\tau_1\right),\dots,x\left(t- \tau_{m-1}\right),\dot x(t)\right),\quad t\in\bbR,\;x(t)\in\mathbf{V} = \mathbb R^n,\\
 x(t)&=x(t+2\pi),\quad\dot{x}(t) = \dot{x}(t+2\pi)
 \end{aligned}\right.
 \end{align}
(where $\tau_k:=\frac{2\pi k}m$, $k=1,2,\dots, m-1$)
 under the following assumption on $f : \bfV \times \bfV^{m-1} \times \bfV \to \bfV$ providing the time reversibility of system \eqref{eq:1}:

\medskip
 \begin{itemize}
\item[(R)] $f(x,y^1,y^2, \cdots, y^{m-2},y^{m-1},z)=f(x,y^{m-1},y^{m-2}, \cdots, y^2, y^1,z)$ 
for all  $(x,y^1,\cdots,y^{m-1},z)\in \bfV^{m+1}$
\end{itemize}

Assume, in addition, that  $\bfV$ is an orthogonal representation of a finite group $\Gamma$. 
Put $\bold{u}:=(x,y^1,\cdots,y^{m-1},z)\in \bfV^{m+1}$ and define on $\bfV^{m+1}$ the diagonal $\Gamma$-action by 
\[\gamma \bold{u} := (\gamma x, \gamma y^1,...,\gamma y^{m-1},\gamma z).\]
 We make the following symmetry and regularity assumptions: 
 \begin{enumerate}[label=($A_\arabic*$)]
  	\item\label{c1} $f$ is $\Gamma$-equivariant, i.e., $f$ is continuous and $f(\gamma \bold{u})=\gamma f(\bold{u})$ for all $\gamma\in\Gamma$ and $\bold{u}\in \bfV^{m+1}$;
  	\item\label{c2}  for all $x,z\in \bfV$ and $\bold y\in \bfV^{m-1}$, one has:
        	\begin{itemize}
	\item[(i)] $f(x,\bold y,-z)=f(z,\bold y,z)$,
	\item[(ii)] $f(-x,-\bold y,z)=-f(z,\bold y,z)$;
	\end{itemize}
	\item\label{c3} The derivative $A:=Df(0)=[A_0,A_1,\dots, A_{m-1},0]$ exists and $A_jA_s =A_sA_j$ for $j,s =0,1,\dots, m-1$.
 \end{enumerate}

Furthermore, we will be looking for periodic solutions ``living" in a prescribed compact $\Gamma$-invariant domain. More formally, 
let  $\eta:\bfV\to \br$ be a function such that:
\begin{itemize} 
\item[($\eta_1$)] $\eta$ is $C^2$-smooth;
\item[($\eta_2$)] $\eta(\gamma x) = \eta(x)$ for all $x\in \bfV$ and $\gamma \in \Gamma$; 
\item[($\eta_3$)]  $\eta(-x)=\eta(x)$ for all $x\in \bfV$; 
\item[($\eta_4$)] $\eta(0)<0$;
\item[($\eta_5$)] $0$ is a regular value of $\eta$;
\item[($\eta_6$)] there exists $R > 0$ such that  $D:=\eta^{-1}(-\infty,0) \subset B_R(0)$, where $B_R(0)$ stands for the open ball of radius $R$ centered at the origin.
\end{itemize}   
Clearly, $\overline{D}$ is a smooth compact (oriented) $\Gamma$-invariant manifold with boundary 
\begin{equation}\label{eq:boundary-M}
C := \partial  D = \eta^{-1}(0)
\end{equation} being 
a smooth $\Gamma$-submanifold  of $\bfV$. Moreover, $-\overline{D}=\overline{D}$ and $0\in D$.
 
A starting point for our discussion is the work  \cite{Amster}, where the authors considered (non-equivariant) non-autonomous systems 
without delays. As a matter of fact, the results obtained in \cite{Amster}, being applied in the autonomous setting, 
do not guarantee that the detected periodic solutions are {\it non-constant}. 
At the same time, by  combining the reversibility of the system in question with other symmetries, we are able to refine the results of \cite{Amster} in such a way that the existence of {\it non-constant} periodic solutions together with their symmetric classification can be provided.

Following \cite{Amster}, 
we will use the concept of second fundamental form in order to formulate 
curvature/growth conditions on $f$ generalizing the classical  Hartman-Nagumo conditions originally formulated for $D = B_R(0)$ (cf.  \cite{Hart,Nagumo}).
Recall the definition of the second fundamental form associated with $C$.  For every  $x\in C$, denote by $n_x$ the outer normal vector to $C$ at $x$ i.e. \begin{equation}\label{eq:n-x}
n_x=\frac{\nabla\eta(x)}{|\nabla\eta(x)|},
\end{equation} and let
$\nu : C\to S^{n-1}$ be the {\it Gauss map} given by $\nu(x):=n_x$. Obviously, for any $x \in C$, the tangent spaces $T_x(C)$ and $T_{n_x}(S^{n-1})$ are parallel, and as such can be identified. This way, for any $x \in C$, the tangent map $d\nu_x$ (as well as its negative known as a {\it Weingarten map} or {\it shape
operator} (see, for example, \cite{Thorpe})) can be considered as a linear map from $T_x(C)$ into itself. The function $\kappa(x):= \det(-d\nu(x))$ is called the {\it Gauss curvature} of $C$. It is well-known (and easy  to see) that $-d\nu_x$ is a self-adjoint operator with respect to the standard inner product $\langle\cdot,\cdot\rangle$ in $\mathbb R^n =  \bfV$. The quadratic form associated with $-d\nu_x$ and denoted $\mathbb I_x(v):= -\langle d\nu_x(v),v \rangle$ is called the {\it second fundamental form} of $M$. We will use the notation $\mathbb I_x(v,w)$ for the bilinear form associated with $\mathbb I_x(v)$.   In particular, for two smooth curves $c$, $d:(-\ve,\ve)\to C$, $c(0)=d(0)=x$ and $\dot c(0) = v$, $\dot d(0)=w$, one has 
\begin{equation}\label{eq:2}
\mathbb I_x(v,w)= - \left\langle \tfrac{d}{dt} \nu(c(t)),\dot d(t)\right\rangle\Big|_{t=0}.
\end{equation}
  
We are now in a position to formulate curvature/growth conditions on $f$ (cf. \cite{Amster}; see also \cite{BebernesSchmitt,GainesMawhin1,GainesMawhin2,Mawhin}):

\begin{enumerate}[label=($A_\arabic*$)]
\setcounter{enumi}{3}
\item\label{c4} for any $x \in C$,  $\bold y \in \bfV^{m-1}$ and $z \in \bf V$ such that $|\bold y|\le R$  and $z \perp n_x$, one has
\begin{equation}\label{eq:A4}
\langle f(x,\bold y,z) ,n_x \rangle > \mathbb I_x(z)
\end{equation}
(cf. ($\eta_1$)--($\eta_6$), \eqref{eq:boundary-M} and \eqref{eq:n-x});

\item\label{c5}
there exist constants $A$, $B>0$ such that the function  $\phi(s) := A + Bs^2$, $s\in \br$, satisfies  
\[
 |f(x, \bold y,z)|\le \phi(|z|)
 \]
for any $(x, \bold y, z) \in \bfV \times \bfV^{m-1} \times \bf V$ with  ${|x|, |\bold y| \le R}$;

\item\label{c6} there exists a constant $K > 0$ such that for any $(x, \bold y, z) \in \bfV \times \bfV^{m-1} \times \bf V$ with ${|x|, |\bold y| \le R}$, 
one has 
\[
|f(x,\bold y,z)|\le \nabla^2\eta(x)(z,z) 
+\langle f(x,\bold y,z),\nabla \eta(x)\rangle + K. 
\]
\item[($A_6'$)] There are constants $\alpha>0$, $K>0$ such that
\[
\forall_{|x|\le R}\;\forall_{|\bold y|\le R}\;\forall_{z\in V}\;\;\;\; |f(x,\bold y,z)|\le \alpha (\langle x, f(x,\bold y, z)\rangle +|z|^2) +K 
\]

\end{enumerate}
\noindent
Given $\eta$ satisfying ($\eta_1$)--($\eta_6$) and $f$ satisfying ($R$) along with  \ref{c1}--\ref{c6} (or ($R$) along with \ref{c1}--\ref{c5} and {\rm ($A_6'$)}), the {\it goal} of  the present paper is to study the existence and spatio-temporal patterns of solutions to problem \eqref{eq:1} living in $D$.  Some  remarks are in order:

(i) Under the assumptions  that $f$ is continuos and satisfies \ref{c4}--\ref{c6}, problem \eqref{eq:1} was considered for {\it non-autonomous ODEs} in \cite{Amster}, with  {\it no symmetry} conditions on $f$ and $D$ being imposed. The method we are using in the present paper allows us to treat equivariant non-autonomous DDEs  the same way as the autonomous ones with cosmetic modifications only. On the other hand, equivariant autonomous systems satisfying condition ($R$) allow us to study
the impact of the orthogonal group $O(2)$ on spatio-temporal patterns of periodic solutions (versus $D_1 = \{1,\kappa\} < O(2)$ in the non-autonomous case).  Also, one can easily adopt the method to treat BVPs rather than periodic problems.  

(ii) Since $\mathbb I_x(z) \geq - \lambda_{min}(x)$ for every $x \in C$ and $z \perp n_x$ (here $\lambda_{min}(x)$ stands for the minimal eigenvalue
of the self-adjoint operator $d\nu_x$), one can replace condition \ref{c4} by the more verifiable one:

\smallskip
\noindent
($A_4^{\prime}$)  for every $x \in C$,  $\bold y \in \bfV^{m-1}$ and $z \in \bf V$ such that $|\bold y|\le R$  and $z \perp n_x$, one has
\begin{equation}\label{eq:A4-prime}
\langle f(x,\bold y,z) ,n_x \rangle \geq -\lambda_{min}(x).
\end{equation}

 \vs 

 {\bf (b) Method.} Observe that given an  orthogonal $G$-representation  $V$ (here  $G$ stands for a compact Lie group) and an admissible $G$-pair  $(f,\Omega)$  in $V$
(i.e. $\Omega\subset V$ is an open bounded $G$-invariant set and $f:V\to V$ is a $G$-equivariant map without zeros on $\partial \Omega$), the Brouwer degree $d_H:=\deg(f^H,\Omega^H)$ is 
well-defined for any $H \le  G$ (here $\Omega^H:= \{x \in \Omega\, :\, hx = x\; \forall h \in H\}$ 
and $f^H:= f|_{\Omega^H}$). If for some $H$, one has $d_H\not=0$, then the existence of solutions with symmetry at least $H$ to equation $f(x)=0$ in $\Omega$, can be predicted. Although this approach provides a  way to determine the existence of solutions in $\Omega$, and even to distinguish their different orbit types, nevertheless, it comes at a price of elaborate $H$-fixed-point space computations which can be a rather challenging task.

Our method is based on the usage of the Brouwer equivariant degree theory; 
 for the detailed exposition of this theory, we refer to the monographs
 \cite{AED, KW,IV-book,KB} and survey \cite{survey} (see also \cite{BKLN,BLN,BHKX}). 
In short, the equivariant degree is a topological tool allowing ``counting'' orbits of solutions to symmetric equations in the same way as the usual Brouwer degree does, but according to their symmetry properties. 

To be more explicit, the equivariant degree $\gdeg(f,\Omega)$ is an element of the free $\bz$-module $A(G)$ generated by the conjugacy classes $(H)$ of subgroups $H$ of $G$ with a finite Weyl group $W(H)$:
\begin{equation}\label{eq:gdeg}
\gdeg(f,\Omega)=\sum_{(H)} n_H\, (H), \quad n_H\in \bz,
\end{equation} 
where the coefficients $n_H$ are given by the following Recurrence Formula
\begin{equation}\label{eq:rec}
n_H=\frac{d_H-\sum_{(L)>(H)} n_L \,n(H,L)\, |W(L)|}{|W(H)|},
\end{equation}
and  $n(H,L)$ denotes the number of subgroups $L'$ in $(L)$ such that $H\le L'$ (see \cite{AED}).  One can immediately recognize a connection 
between the two collections: $\{d_H\}$ and  $\{ n_H\}$, where $H \le  G$ and $W(H)$ is finite. 
As a matter of fact,  $\gdeg(f,\Omega)$ satisfies the standard properties expected from any topological degree.  
However, there is one additional functorial property, which plays a crucial role in computations, namely the {\it product property}. In fact, $A(G)$ has a natural structure of a ring  (which is called the {\it Burnside ring} of $G$), where the multiplication  $\cdot:A(G)\times A(G)\to A(G)$   is defined on generators by 
\begin{equation}\label{eq:mult}
(H)\cdot (K)=\sum_{(L)} m_L\, (L) \quad\quad (W(L) \text{ is finite}),
\end{equation}
where the integer $m_{L}$ represents the number of $(L)$-orbits contained in the space $G/H\times G/K$ equipped with the natural diagonal $G$-action.
The product property for two admissible $G$-pairs $(f_1,\Omega_1)$ and   $(f_2,\Omega_2)$ means the following equality:
\begin{equation}\label{eq:mult-property}
 \gdeg(f_1\times f_2,\Omega_1\times \Omega_2)=  \gdeg(f_1,\Omega_1)\cdot  \gdeg(f_2,\Omega_2).
\end{equation}
Given a $G$-equivariant linear isomorphism $A : V \to V$, formula \eqref{eq:mult-property} combined with the equivariant spectral decomposition of $A$, reduces the computations of $\gdeg(A,B(V))$ to the computation of the so-called basic degrees $\deg_{\cV_k}$, 
which can  be `prefabricated'
in advance for any group $G$ (here $\deg_{\cV_k}:=\gdeg(-\id, B(\cV_k))$ with $\cV_k$ being an irreducible $ G$-representation and $B(X)$ stands for the unit ball in $X$).  ln many cases, the equivariant degree  based method can be easily assisted
by computer (its usage seems to be unavoidable for large symmetry groups). 

In the present paper, to establish the abstract results on the existence and symmetric properties of periodic solutions,
we use the  $G$-equivariant Brouwer degree with $G:=O(2)\times \Gamma\times \bz_2$, where
$O(2)$ is related to the reversal symmetry combined with 
the autonomous form of the system, $\Gamma$ reflects symmetries of $D$ and/or possible coupling  in the corresponding network of identical oscillaltors, and $\mathbb Z_2$ is related to the oddness of $f$. We also present a concrete illustrating example with $\Gamma:= D_8$, where $D_8$ stands for the dihedral group of order 16. Our computations are essentially based on new group-theoretical computational algorithms, which were implemented in the specially created Hao-Pin Wu (see \cite{Pin}) package {\it EquiDeg} for the  GAP system.

 \vs

 {\bf (c) Overview.} After the Introduction, the paper is organized as follows. 
 In Section \ref{sec:a-priori}, we establish a priori bounds for solutions to problem \eqref{eq:1} in the space $C^2(S^1;\bfV)$ (actually, we assume that values of solutions ``live" in a given open bounded symmetric domain $D \subset \bfV$; cf. \eqref{eq:x-t-form} for the precise formulation). In Section 
 \ref{sec:operator-reform}, we reformulate problem \eqref{eq:x-t-form} as an 
 $O(2) \times \Gamma \times \mathbb Z_2$-equivariant fixed point problem  in $C^2(S^1;\bfV)$ and present an {\it abstract} equivariant degree based result. This result can be effectively applied to concrete symmetric systems only if a ``workable" formula for the degrees associated can be elaborated. The latter is a subject of Sections \ref{sec:degree-computation} and \ref{sec:degree-computation-D}. In Section \ref{sec:degree-computation},  we combine the product property of the equivariant degree with equivariant spectral data of the linearization of the operator equation at the origin in order to reduce the degree computations to products of appropriate basic degrees. In Section \ref{sec:degree-computation-D}, we compute the degree of the operator involved on the boundary of the domain provided by the a priori bound. Actually, this is the place where the curvature of $\partial D$ and $G$-equivariant degree come together: here we essentially use admissible homotopies considered in \cite{Amster}.
In Section \ref{sec:main-results-example}, based on the results of Sections \ref{sec:a-priori}--\ref{sec:degree-computation-D}, we present our main results (see Theorems \ref{th:main1} and \ref{th:main2}) expressed in terms of the function $\eta$ (cf. ($\eta_1$)--($\eta_6$)) and right-hand side of  
\eqref{eq:x-t-form} only. 
As an example, we consider $\bfV = \mathbb R^2$ equipped with the natural $\Gamma:= D_8$-representation and explicitly describe a $D_8$-invariant function $\eta: \bf V \to \mathbb R$ giving rise to the $D_8$-invariant domain $D$ with $\partial D$ admitting points with both positive and negative curvature. 
Using $\nabla \eta$, we explicitly describe $f$ in  \eqref{eq:x-t-form} satisfying
(R), \ref{c1}--\ref{c5}, {\rm ($A_6'$)}. We conclude the paper with an Appendix related to amalgamated notation for subgroups of group products, 
equivariant topology jargon and equivariant degree background.

  \section{A Priori Bound and $C$-touching}\label{sec:a-priori}
  
   \subsection{A priori bound for the first derivative}
   
In this subsection, we establish  a priori bounds for the first and second derivatives of solutions to problem \eqref{eq:1} living in $D$. The lemma following below  can be traced back to \cite{Hart}, where the case of ODEs was studied for $D = B_R(0)$. In our proof, we combine the ideas from \cite{Amster} (where Hartman's result was extended to arbitrary $D$)  with \cite{BHKX} (where the case of equivariant  DDEs and $D = B_R(0)$ was considered). To simplify our notations,
given a function $x : \mathbb R \to \bfV$, put $\bold{x}_t:= \left(x\left(t-\tau_1\right),\dots,x\left(t- \tau_{m-1}\right)\right)$, so that we are interested in the problem 
\begin{align}\label{eq:x-t-form}
\left\{  \begin{aligned}
 \ddot{x}(t)&=f\left(x(t), \bold{x}_t, \dot x(t)\right),\quad t\in\bbR,\;x(t)\in \overline{D} \subset \mathbf{V} = \mathbb R^n,\\
 x(t)&=x(t+p),\quad\dot{x}(t) = \dot{x}(t+p),
 \end{aligned}\right.
 \end{align}
 where $p:= 2\pi$.

   \vs 
   \begin{lemma}\label{lm:2} Let $\eta : \bfV \to \mathbb R$ satisfy ($\eta_1$), ($\eta_4$)--($\eta_6$), and let $f : \bfV \times \bfV^{m-1} \times \bfV \to \bfV$
   be a continuous map satisfying \ref{c5} and \ref{c6}  (resp. \ref{c5} and {\rm ($A_6'$)}). If  $x = x(t)$ is a solution to \eqref{eq:x-t-form} such that $|x(t)|\le R$ for $t\in \br$, then there exists a constant $M:=M(\phi,\eta, K,p,R)$ (resp.  $M:=M(\phi,\alpha, K,p,R)$) such that 
	\begin{equation}\label{eq:ap-2}
	\forall_{t\in \br}\;\;\;|\dot x(t)|\le M.
\end{equation}
	\end{lemma}
  \begin{proof} We only prove Lemma \ref{lm:2} assuming that  $f$ satisfies \ref{c5} and \ref{c6}.  The case when $f$ satisfies \ref{c5} and {\rm ($A_6'$)} was treated in \cite {BKLN} (see also Remark \ref{rem:eq:M-from-Xiaoli}).
 
 \vs
 Let $x = x(t)$ be a $C^2$-smooth solution to \eqref{eq:x-t-form}. Since $|x(t)|\le R$, one has $|x(t-\tau_k)|\le R$ 
  for all $k= 1,..., m-1$,
  so that $|\bold{x}_t| \leq R$. 
	Put $\boldsymbol \eta(t):=\eta(x(t))$, $t\in \br$. Then by \ref{c6},  one has 
	\begin{align*}
	|\ddot x(t)| &= |f(x(t), \bold x_t,\dot x(t))|\le  \nabla^2\eta(x(t))(\dot x(t),\dot x(t)) +\left\langle f(x(t),\bold x_t,\dot x(t)), \nabla \eta(x(t))\right\rangle +K\\
	& = \boldsymbol  \eta''(t)+K.
	\end{align*}
	Thus, 
	\begin{equation}\label{eq:ap-3}
	\forall_{t\in \br}\;\; \;\; |\ddot x(t)|\le  \boldsymbol\eta''(t)+K.
	\end{equation}
	
			Next, by using integration by parts and the fact that $x(t)$ is $p$-periodic, one calculates: 
\begin{align*}
 \int_t^{t+p} (t+p-s)\ddot x(s)ds&=  (t+p-s)\ddot x(s)\Big|_{t}^{t+p} +\int_{t}^{t+p} \dot x(s)ds= x(t+p)-x(t) - p \dot x(t)= -p\dot x(t) 
 \end{align*}
i.e. 
\begin{equation}\label{eq:ap-4}
\forall_{t\in \br}\;\;\;   p\dot x(t)= - \int_t^{t+p}\;\; (t+p-s)\ddot x(s)ds.
\end{equation}
Similarly,
\[
p\dot x(t) = x(t)-x(t-p)-\int_{t-p}^t (t-p-s)\ddot x(s)ds=-\int_{t-p}^t (t-p-s)\ddot x(s)ds,\]
i.e.
\begin{equation}\label{eq:ap-5}
 p\dot x(t)=-\int_{t-p}^t (t-p-s)\ddot x(s)ds.
\end{equation}
Then by \eqref{eq:ap-4}, one obtains
\[
p\dot x(0)\; =\;-\int_0^{p} \; (p-s)\ddot x(s)ds,
\]  
and by \eqref{eq:ap-3} and $p$-periodicity of $x$, one has:
\begin{align*}
p|\dot x(0)|\;
&\le \;\int_0^{p}\;(p-s)|\ddot x(s)|ds\le \; \int_0^{p}\;(p-s)\Big(\boldsymbol\eta''(s)+K\Big)ds\\
&=\int_0^{p}\;(p-s)\boldsymbol \eta''(s)ds+K\int_0^{p}\;(p-s)ds = -p \boldsymbol\eta'(0)+\frac {1}{2}Kp^2,
\end{align*}
i.e.
\begin{equation}\label{eq:ap-6}
 p|\dot x(0)|\; \le  -p\boldsymbol \eta'(0)+\frac {1}{2}Kp^2.
\end{equation}
Similarly, by \eqref{eq:ap-5}, one obtains
\begin{equation}\label{eq:ap-7}
p|\dot x(0)| \le p\boldsymbol \eta' (0)+\frac {1}{2}Kp^2.
\end{equation}
By adding inequalities \eqref{eq:ap-6} and \eqref{eq:ap-7},  one obtains 
\begin{equation}\label{eq:ap-0}
2p|\dot x(0)|\le Kp^2\quad \Leftrightarrow\quad 
|\dot x(0)|\le \frac {1}{2}Kp.\end{equation}
Moreover (see \eqref{eq:ap-4} and \eqref{eq:ap-3}), one has:
\begin{align*}
p|\dot x(t)|
&\le \int_t^{t+p}\;(t+p-s)|\ddot x(s)|ds\le \int_t^{t+p}\;(t+p-s)\Big(\boldsymbol\eta''(s)+K\Big)ds= - p\boldsymbol \eta'(t)+\frac {1}{2}Kp^2.
\end{align*}
The last inequality, together with condition \ref{c5} imply 
\begin{equation} \label{eq:ap-8}
\frac {\langle \dot x(t),   \ddot x(t) \rangle}{\phi(|\dot x(t)|)}\le \frac {|\langle \dot x(t) ,  \ddot x(t)\rangle|} {\phi(|\dot x(t)|)}\le \frac{|\dot x(t)||\ddot x(t)|}{\phi(|\dot x(t)|)}\le |\dot x(t)|\le \frac {1}{2}Kp-\boldsymbol \eta'(t).
\end{equation}
Next, by integrating inequality \eqref{eq:ap-8}, one obtains for $t\in [0,p]$:
\begin{equation}\label{eq:integr-ineq}
\left|\int_0^{t}\;\frac {\langle \dot x(s) ,  \ddot x(s)\rangle}{\phi(|\dot x(t)|}ds\right|
 \le \int_0^{t}\left[\frac {1}{2}Kp-\boldsymbol \eta'(s)\right]ds 
=\frac {1}{2}Kpt-\boldsymbol \eta(t)+\boldsymbol \eta(0)\le \frac {K}{2}p^2+2\wt R, 
\end{equation}
where $\wt R:=\max\{|\eta(x)|: x\in \overline D\}$.
On the other hand, by making substitution $u = |\dot{x}(s)|$, one obtains:
\begin{equation}\label{eq:integr-rep}
\int_0^{t}\;\;\frac {\langle \dot x(s) , \ddot x(s)\rangle}{\phi (|\dot x(s)|)}ds=\int_{ |\dot x(0)|}^{|\dot x(t)|}\frac {udu}{\phi (u)}.
\end{equation}
Put $\displaystyle \Phi (w):=\;\int_0^{w}\;\frac {udu}{\phi(u)}$,  then
\begin{equation}\label{eq:estim-Phi}
\left|\int_{ |\dot x(0)|}^{|\dot x(t)|}\;\;\frac {udu}{\phi (u)}\right|= |\Phi(|\dot x(t)|)-|\Phi(|\dot x(0)|)|.
\end{equation}
Therefore (cf. \eqref{eq:integr-ineq}--\eqref{eq:estim-Phi}),
\begin{align*}
 |\Phi(|\dot x(t)|)-\Phi (|\dot x(0)|)|\le \frac{K}{2}p^2+2\wt R, 
\end{align*}
in particular,
\begin{equation}\label{eq:module-ineq}
\Phi (|\dot x(t)|) \le \frac {1}{2}Kp^2+2\wt R+\Phi (|\dot x(0)|)|. 
\end{equation}
By \ref{c5},
$\displaystyle \lim_{w \to \infty} \Phi(w) = \infty$, hence,  $\Phi:[0,\infty)\to [0,\infty)$ is a continuous monotonic bijective function.  
Therefore (see  \eqref{eq:ap-0} and \eqref{eq:module-ineq}), the inequality
\begin{equation*}
\Phi (|\dot x(t)|)   
\le \frac{1}{2}Kp^2+2\wt R+\Phi\left(\frac{1}{2}Kp\right),
	 \end{equation*}
	 implies	 
\begin{equation}\label{estim-M}	
|\dot x(t)|\le \Phi^{-1}\left[\frac{1}{2}Kp^2+2\wt R+\Phi\left(\frac{1}{2}Kp\right)\right]=:M,
\end{equation}
and the required estimate follows.
\end{proof}

\begin{remark}\label{rem:eq:M-from-Xiaoli}
Observe that if \ref{c6} is replaced with {\rm ($A_6'$)}, then the following estimate for $\dot x(t)$ was established in \cite {BKLN}:
\begin{equation}\label{eq:M-from-Xiaoli}
|\dot x(t)|\le \Phi^{-1}\left[\frac{1}{2}Kp^2+\alpha R^2+\Phi\left(\frac{1}{2}Kp\right)\right]=:M
\end{equation}
\end{remark}

One has the following immediate consequence of Lemma \ref{lm:2}.

\begin{lemma}\label{lem:second-der-est}
Under the assumptions of Lemma \ref{lm:2}, there exists $N > 0$ such that for any $C^2$-smooth solution $x= x(t)$ to \eqref{eq:x-t-form},
one has 
\begin{equation}\label{eq:esti-N}
\forall_{t \in \mathbb R} \quad  \left| \ddot{x}(t) \right| \le N.
\end{equation}
\begin{proof}
Let $M$ be a constant provided by Lemma \ref{lm:2}. Put
\begin{equation}\label{eq:N}
N:= \max_{x\in \overline{D},\;  \bold y\in \overline{D}^{m-1}, \; |z| \le M}|f(x,\bold y,z)|.
\end{equation}
Let $x = x(t)$ be a $C^2$-smooth solution to \eqref{eq:x-t-form}. Then,
\begin{equation*}
\forall_{t \in \mathbb R} \quad \left| \ddot{x}(t)\right| = \left| f(x(t), \bold{x}_t, \dot{x}(t)\right| \le N.
\end{equation*}  
\end{proof} 

\end{lemma}

\subsection{$C$-Touching}
In this subsection, essentially following \cite{Amster} and \cite{do Carmo}, we show that solutions to problem \eqref{eq:x-t-form} cannot touch 
$C:= \partial D$ provided that $f$ satisfies \ref{c4}. More precisely, 
\begin{lemma}\label{lm:1}
Let $\eta : \bfV \to \mathbb R$ satisfy ($\eta_1$), ($\eta_4$)--($\eta_6$), and let $f : \bfV \times \bfV^{m-1} \times \bfV \to \bf V$ be a continuous map satisfying \ref{c4}. Let  $x:\br\to \bfV$ be a $C^2$-smooth $2\pi$-periodic function such that:
\begin{itemize}
\item[(i)]  $x(t) \in \overline{D}$ for all $t\in \br$;

\item[(ii)] $x(t_o)\in C$ for some $t_o\in \mathbb R$. 
\end{itemize}
Then,  $x$ is not a solution to  problem \eqref{eq:x-t-form}. 

\end{lemma}
 
\begin{proof} Assume for contradiction that $x:\br\to \bfV$ is a $C^2$-smooth $2\pi$-periodic solution to  problem \eqref{eq:x-t-form} satisfying (i) and (ii).
Take a tubular neighborhood of $C$ around the point $x(t_o)\in C$. Then, for a sufficiently small $\ve > 0$, one can represent $x$ as follows: 
\begin{equation}\label{eq:tub-neghb}
x(t)=\alpha(t) +\beta(t)\nu(\alpha(t)) \quad\quad\quad 
 (\alpha(t)\in C, \;\; \beta(t)\le 0, \;\; t \in (t_o-\ve,t_o+\ve)).
\end{equation}
Combining \eqref{eq:tub-neghb} with the fact that $x = x(t)$ is a solution to  \eqref{eq:x-t-form} and using $n_{\alpha(t)} \perp \dot{\alpha}(t)$ and 
$n_{\alpha(t)} \perp {d \over dt}\left(\nu(\alpha(t))\right)$, 
one obtains:
\begin{align}
\langle f(x(t),\bold x_t,\dot x(t)) ,  n_{\alpha(t)} \rangle &= \langle \ddot x(t) ,  n_{\alpha(t)} \rangle =\tfrac d{dt} \langle \dot x(t) ,  n_{\alpha(t)}\rangle  - \langle \tfrac d{dt} \nu(\alpha(t)) ,  \dot x(t)\rangle \notag\\
&= {d \over dt} \Big\langle {d \over dt} \Big[ \alpha(t) +\beta(t)\nu(\alpha(t))\Big], n_{\alpha(t)} \Big\rangle   - \langle \tfrac d{dt} \nu(\alpha(t)) ,  \dot x(t)\rangle \notag\\
&= {d \over dt} \Big\{\langle \dot{\alpha}(t),  n_{\alpha(t)} \rangle + \dot{\beta}(t) \langle \nu(\alpha(t),  n_{\alpha(t)} \rangle + 
\beta(t) \langle \tfrac d{dt} \nu(\alpha(t)),  n_{\alpha(t)}\rangle \Big\}  - \langle \tfrac d{dt} \nu(\alpha(t)) ,  \dot x(t)\rangle\notag\\
&= {d \over dt} \Big\{ \dot{\beta}(t)  \|n_{\alpha(t)}\|^2 \Big\}  - \langle \tfrac d{dt} \nu(\alpha(t)) ,  \dot x(t)\rangle\notag\\
&= \ddot{\beta}(t) - \langle \tfrac d{dt} \nu(\alpha(t)),  \dot x(t)\rangle. \label{eq:beta-est}
\end{align}
Since $\beta(t)$ achieves its local maximum at $t_o$, one has $\ddot{\beta}(t_o)\le 0$. Combining this with  \eqref{eq:beta-est} yields:
\begin{equation*}
\langle f(x(t_o),\bold x_{t_o},\dot x(t_o)) ,  n_{x(t_o)}\rangle \le 
 - \langle \tfrac d{dt} \nu(\alpha(t)) ,  \dot x(t)\rangle\Big |_{t=t_o}
= \mathbb I_{x(t_o)} (\dot x(t_o)),\\
\end{equation*} 
which contradicts condition \ref{c4}.
   \end{proof}  
   \vs

  \section{Operator Reformulation in Function Spaces}\label{sec:operator-reform}
   \subsection{Spaces}
  Denote by  $C_{2\pi}(\bbR;\bfV)$  the space of continuous $2\pi$-periodic functions equipped with the norm
  \begin{equation}
   \|x\|_{\infty}=\sup_{t\in\bbR}|x(t)|,\quad x\in C_{2\pi} (\bbR;\bfV).
  \end{equation}
  Denote by  $\scrE :=C^2_{2\pi}(\bbR,\bfV)$ the  space of $C^2$-smooth  $2\pi$-periodic functions from $\mathbb{R}$ to $\bfV$
  equipped with the norm
  \begin{align}
   \norm{x}_{\infty,2}&=\mbox{max}\{\|x\|_{\infty},\|\dot{x}\|_{\infty},\|\ddot{x}\|_{\infty}\}.
  \end{align}
Let $O(2)$ denote the group of orthogonal $2\times 2$ matrices. Notice that $O(2)= SO(2)\cup SO(2) \kappa$, where $\kappa=\begin{bmatrix}1&0\\0&-1\end{bmatrix}$, and $SO(2)$ denotes the group of rotations $\begin{bmatrix}\cos\tau&-\sin\tau\\\sin\tau&\cos\tau\end{bmatrix} $ which can be identified with $e^{i\tau}\in S^1\subset\mathbb{C}$. Notice that $\kappa e^{i\tau}=e^{-i\tau}\kappa$.\par	
  
Put $G:=O(2)\times \Gamma\times\mathbb{Z}_2$ and define the $G$-action on $\scrE$ by
    \begin{align}
    (e^{i\theta}, \gamma,\pm 1)x(t)&:=\pm\gamma x(t+\theta),\label{action1}\\
    (e^{i\theta}\kappa, \gamma,\pm 1)x(t)&:=\pm\gamma x(- t+ \theta)\label{action2},
  \end{align}
 where $x\in\scrE,\; e^{i\theta},\kappa \in O(2),\; \gamma \in \Gamma$ and $ \pm 1 \in \mathbb Z_2$. Clearly, $\scrE$ is  an isometric  Banach 
 $G$-representation. 
 In a standard way, one can identify a $2\pi$-periodic function $x:\bbR\rightarrow \bf V$ with a function $\tilde{x}: S^1\rightarrow \bfV$, so one can write $C^2(S^1,\bfV)$ instead of $C^2_{2\pi}(\bbR,\bfV)$. Similar to \eqref{action1}-\eqref{action2} formulas define isometric $G$-representations on the spaces 
of periodic functions $C_{2\pi}(\bbR,\bfV)$ and $L^2_{2\pi}(\mathbb R;V)$ to which appropriate identifications are applied. 

Let us describe the $G$-isotypic  decomposition of $\scrE $. Consider, first, $\scrE $ as an $O(2)$-representation
corresponding to its Fourier modes: 
\begin{align}\label{dcp}
\scrE=\overline{\bigoplus\limits_{k=0}^{\infty}\mathbb{V}_k},\quad\mathbb{V}_k:=\{\cos(kt)u+ \sin(kt) v:u,\, v\in \bfV\},
\end{align}
where each $\mathbb V_k$, for $k\in \bn$, is equivalent to the complexification $\bfV^c := \bfV \oplus i \bfV$ (as a {\it real} $O(2)$-representation)  of 
$\bf V$, where the rotations $e^{i\theta}\in SO(2)$ act on vectors $\bold z\in \bfV^c$ by $e^{i\theta}(\bold z) :=e^{-ik\theta}\cdot \bold z$ (here `$\cdot$'  stands for complex multiplication) and $\kappa \bold z:=\overline {\bold z}$. Indeed, the linear isomorphism $\vp_k : \bfV^c\to \mathbb V_k$ given by 
\begin{equation}\label{eq:complexification}
\vp_k(x+iy):= \cos(kt) u + \sin(kt) v, \quad u,\, v\in \bfV,
\end{equation}
is $O(2)$-equivariant. Clearly, $\mathbb V_0$ can be identified with $\bfV$ with the trivial $O(2)$-action, while $\mathbb V_k$, $k = 1,2,\ldots$, 
is modeled on the irreducible $O(2)$-representation $\cW_k\simeq\mathbb{R}^2$, where $SO(2)$ acts by $k$-folded rotations and $\kappa$ acts by complex conjugation. 
 
Next, each $\mathbb V_k$, $k = 0,1,2,\ldots$, is also $\Gamma \times \mathbb Z_2$-invariant. Let  $\cV_0^-, \cV_1^-,\cV_2^-,\dots, \cV_{\mathfrak r}^-$
be a complete list of all irreducible orthogonal $\Gamma \times \mathbb Z_2$-representations on which $\Gamma \times \mathbb Z_2$-isotypic  components of $\bfV \simeq \mathbb V_0$ are modeled (here ``$^-$" stands to indicate the antipodal $\mathbb Z_2$-action and $\cV_0^-$ corresponds to the trivial $\Gamma$-action). Since 
$\cV_{k,l}^- :=\cW_k\otimes \cV_l^-$ is an irreducible orthogonal $G$-representation, it follows that  $\mathbb V_0$ and   $\mathbb V_k$ (cf. \eqref{dcp}) admit the following $G$-isotypic  decompositions:
 \begin{equation}\label{iso-0}
     \mathbb V_0=V_{0}^-\oplus V^-_{1}\oplus\dots \oplus V^-_{\mathfrak r}
 \end{equation}
 (with the trivial $O(2)$-action) and 
    \begin{equation}\label{eq:iso-k}
        \mathbb V_k = V_{k,0}^-\oplus V_{k,1}^-\oplus\dots \oplus V_{k,\mathfrak r}^-,
        \end{equation}
where  $V_l^-$ (resp. $V_{k,l}^-$) is modeled on  $\cV_{0,l}^-$  (resp. $\cV_{k,l}^-$ with $k > 0$).    
\begin{remark}\label{rem:non-constant-solutions}        
Clearly, $x \in C^2(S^1;\bfV)$ is not a constant function if $G_x$ does {\it not} contain $O(2) \simeq O(2) \times \{1\} \times \{1\} < G$.
\end{remark}

  \subsection{Operators} 
   Define the following operators:
   \begin{alignat*}{3}
   \mathfrak i: & \,\scrE \rightarrow C(S^1,\bfV),\quad &(\mathfrak i x)(t) &:= x(t)\\
    L:& \, \scrE \rightarrow C(S^1,\bfV),\quad &(Lx)(t)&:=\ddot{x}(t) - (\mathfrak i x)(t)\\
    j : &\, \scrE\rightarrow C(S^1,\bfV^{m+1}),\quad & (jx)(t)&:=(x(t),x(t-\tau),\dots,x(t-(m-1)\tau),\dot x(t))
    \end{alignat*}
and the Nemytskii operator 
  $N_f :  \;C(S^1,\bfV^{m+1}) \rightarrow  C(S^1,\bfV)$ given by 
 \[
 (N_f (x,\bold y,z))(t) := f(x(t),y^1(t),\dots,y^{m-1}(t),z(t)). 
\]
 The above operators are illustrated  on the (non-commutative) diagram following below: 
   \begin{figure}[h!]\label{pict}
    \centering
  	\begin{tikzpicture}[>=angle 90]
  	
  	\matrix(a)[matrix of math nodes,
  	row sep=5.5em, column sep=2.5em,
  	text height=1.5ex, text depth=2ex]
  	{\scrE& & C(S^1,\bfV)\\
  		& C(S^1,\bfV^{m+1}) \\};
  	\path[->, ]  (a-1-1) edge node[above]{$L,\,{\mathfrak i}$}
  	                                   (a-1-3);
  	\path[->](a-1-1) edge node[below left]{$j$}
  	                      (a-2-2);
  	\path[<-](a-1-3) edge node[below right]{$N_f $} (a-2-2);
  	
  	\end{tikzpicture}
 
  	\caption{Operators involved}
   \end{figure}
  
  \noindent
  System \eqref{eq:1} is equivalent to
   \begin{align}\label{eq:4}
    Lx= N_f  (jx)-\mathfrak i(x),\quad  x\in\scrE.
   \end{align}
  Since $L$ is an isomorphism, equation \eqref{eq:4} can be reformulated as follows:
   \begin{align}\label{eq:5}
    \scrF(x):=x- L^{-1}(N_f (jx)-\mathfrak i(x))=0, \;\; x\in\scrE. 
   \end{align}
   
\begin{proposition}\label{prop}
Suppose that  $f$ satisfies conditions (R), \ref{c1}--\ref{c3}, and  the nonlinear operator $\scrF:\mathscr E\to \mathscr E$ is given by \eqref{eq:5}. Then, 
the map $\mathscr F$ is a $G$-equivariant completely continuous field.
 \end{proposition}
\begin{proof} Combining \eqref{dcp} and  \eqref{eq:complexification} with the definition of $L$ yields:
\begin{equation}\label{eq:action-L}
L|_{\mathbb V_k} =-(k^2+1) \id:V^c\to V^c \quad \rm{and} \quad L|_{\mathbb V_0}=-\id \quad (k > 0). 
\end{equation}
In particular, $L$ (and, therefore, $L^{-1}$) is $G$-equivariant. Since $j$ is the embedding, it is $G$-equivariant as well. Since 
$L$ and $N_f $ are continuous and $j$ is a compact operator, it follows that  $\scrF$ is a completely continuous field. 
Also, by assumption \ref{c1} (resp.  \ref{c2}), 
$\scrF$ is $\Gamma$-equivariant (resp. $\bz_2$-equivariant). 
Since system   \eqref{eq:1} is autonomous, it follows that $\scrF$ is $SO(2)$-equivariant. To complete the proof of part (i), one only needs to show that  $\scrF$ commutes with the $\kappa$-action. 
In fact,  for all $t\in \br$ and $x\in \mathscr E$, one has
(we skip $\mathfrak i$ to simplify notations):  
{\footnotesize 
		\begin{align*}
		\mathscr F(\kappa x)(t)
		&= \kappa x(t)-L^{-1}\Big(f(\kappa x(t), \kappa \bold x_t, \kappa \dot x(t))-\kappa x(t)\Big)\\
		&=x(-t) -L^{-1}\Big( f(x(-t),x(-t + \tau_1), \dots, x(-t + \tau_{m-1})), -\dot x(-t)) - x(-t)\Big) \; \text{(by \eqref{action2})}\\
		&=x(-t) -L^{-1}\Big( f(x(-t),x(-(t  + 2\pi - \tau_1)),\dots, x(-(t + 2\pi - \tau_{m-1})),- \dot x(-t))-x(-t)\Big) \; \text{(by periodicity of $x$)}\\
		&=x(-t) -L^{-1}\Big( f(x(-t),x(-(t  + 2\pi - \tau_1)),\dots, x(-(t + 2\pi - \tau_{m-1})),  \dot x(-t)) - x(-t)\Big) \; \text{(by (A2)(i))} \\
		&=x(-t) -L^{-1}\Big( f(x(-t),x(-t-\tau_{m-1}),\dots, x(-t-\tau_1),\dot x(-t))-x(-t)\Big) \; \text{(by choice of $\tau_k$)} \\
		&=x(-t) -L^{-1}\Big( f(x(-t),x(-t-\tau_1),\dots, x(-t-\tau_{m-1}),\dot x(-t))-x(-t)\Big)  \; \text{(by (R))} \\
		&=\kappa x(t)-\kappa L^{-1}\Big( f(x(t),x(t-\tau_1), \dots, x(t-\tau_m),\dot x(t))-x(t)\Big)\; \text{(by \eqref{action2})}\\
		&=\kappa\Big(x(t)-L^{-1}\big(f(x(t),\bold x_t,\dot x(t)) -x(t)\Big)\\
		&=\kappa \mathscr F(x)(t).
		\end{align*}
		}\end{proof}

\vs
\begin{remark}\label{rem:propert-scr-A}
Notice that if $f$ satisfies \ref{c2}(ii), then $x(t)\equiv 0$ is a solution to equation \eqref{eq:5}. Also,  
the operator
 \begin{equation}\label{eq:linearization}
     \scrA:=D\scrF(0) : {\scrE} \longrightarrow {\scrE}
 \end{equation}
 is correctly defined provided that condition \ref{c3} is satisfied. Moreover, in this case, 
   \begin{equation}\label{eq:linearization-formula}
    \scrA=\id-L^{-1}\big(DN_f (0) \circ j - \mathfrak i\big) : \scrE\longrightarrow\scrE 
   \end{equation}
 and $\scrA$ is a Fredholm operator of index zero; in particular,
 $\scrA$ is an isomorphism if and only if  $0\not\in \sigma(\scrA)$. 
 Furthermore, if $f$ satisfies  \ref{c1}--\ref{c3}, then the $G$-equivariance of $\scrF$ together with $G(0) = 0$ imply the  $G$-equivariance of $\scrA$. 
 \end{remark}

 \vs
 We will also need the following lemma (its proof is standard and can be found in \cite{BHKX}).
   
 \begin{lemma}\label{lm:3}
 Under the assumptions (R), \ref{c1}-\ref{c3}, suppose, in addition, that $0 \not\in \sigma(\scrA)$ (here $\sigma(\scrA)$ stands for the spectrum of $\scrA$)
 (cf. \eqref{eq:5} and \eqref{eq:linearization}-\eqref{eq:linearization-formula}).
Then, for a sufficiently small $\varepsilon>0$, 
the map $\scrF$ is $B_{\varepsilon}(0)$-admissibly  $G$-equivariantly homotopic to $\scrA$. 
\end{lemma}

 \vs
\subsection{Abstract equivariant degree based result} 
	
Assuming that conditions ($\eta_1$)--($\eta_6$), (R) and \ref{c1}--\ref{c6} (resp. ($\eta_1$)--($\eta_6$), (R), \ref{c1}--\ref{c5} and {\rm ($A_6'$)}) are satisfied, we are going  to formulate an equivariant degree based result related to problem \eqref{eq:x-t-form}. To this end, one needs: (i) to construct an open bounded $G$-invariant 
domain $\Omega \subset  \mathscr E$, $0 \in \Omega$,  such that $\mathscr F(x)$ is  $\Omega$-admissible,  and (ii) to introduce 
additional concepts related to {\it maximality} of orbit types.

\vs 
Take $\phi$  from assumption \ref{c5} and $K$ from assumption \ref{c6} (resp.{\rm ($A_6'$)}). With an eye towards deforming $\mathscr F$ by an $\Omega$-admissible $G$-homotopy and to be on the safe side, take  $M := M(2 \phi,\eta, K+1,p,R)$  (resp.  $M := M(2 \phi,\alpha, K+1,p,R)$ 
provided by Lemma \ref{lm:2} (resp. Remark \ref{rem:eq:M-from-Xiaoli})
Next, take $N > 0$ provided by Lemma \ref{lem:second-der-est} and  put
 \begin{equation}\label{eq:Omega-set}
 \Omega:=\left\{x\in \mathscr E: \forall_{t\in \br} \;\; x(t)\in D, \; \|\dot x\|_\infty<M+1, \; \|\ddot x\|_\infty<N+1\right\}
 \end{equation}
 (see ($\eta_6$) for the definition of $D$). 
 It is easy to see that $\Omega$ is an open bounded and  $G$-invariant set. Moreover, 
 \vs
 \begin{lemma}\label{lm:4}  Under the assumptions ($\eta_1$)--($\eta_6$), (R) and \ref{c1}--\ref{c6}  (resp. ($\eta_1$)--($\eta_6$), (R), \ref{c1}--\ref{c5} and {\rm ($A_6'$)}), the map $\mathscr F$ (given by \eqref{eq:5}) is $\Omega$-admissible.
 \end{lemma} 
 \begin{proof} 
Suppose for contradiction, that there exists $x\in \partial \Omega$ such that $\mathscr F(x)=0$. Then, there exists a sequence $\{x_n\} \subset \Omega$
such that $\|x_n - x\|_{\infty,2} \to 0$ and $x \not\in \Omega$. In particular (see \eqref{eq:Omega-set}),
\begin{equation}\label{eq-translation}
\forall_{n\in \mathbb N} \; \forall_{t\in \mathbb R} \quad x_n(t) \in D \subset \overline{D}. 
\end{equation}
Combining \eqref{eq-translation} with the uniform convergence yields 
\begin{equation}\label{eq:translate2}
\forall_{t \in \mathbb R}  \quad x(t) \in \overline{D}. 
\end{equation}
Since $\mathfrak F (x) = 0$, relation  \eqref{eq:translate2} together with 
Lemmas \ref{lm:2} and \ref{lem:second-der-est} imply:
\begin{equation}\label{eq:rememb-L.2.1}
\|\dot x\|_\infty \leq M < M+1  \qquad \text{and} \qquad \|\ddot{x}(t) \|_\infty \le N < N +1.
\end{equation}
Since $x \not\in \Omega$, inequalities \eqref{eq:rememb-L.2.1} together with \eqref{eq:Omega-set} imply that there exists $t_o \in \mathbb R$
such that $x(t_o) \not\in D$, hence (see again \eqref{eq:translate2}), $x(t_o) \in C:= \partial D$, but this contradicts Lemma \ref{lm:1}.
 \end{proof}

Observe that under the assumptions of Lemma \ref{lm:4}, the $G$-equivariant degree $\gdeg(\mathscr F,\Omega)$ is well-defined. Also, under the assumptions of Lemma \ref{lm:3},  $\gdeg(\scrA,B_{\varepsilon}(0))$ is well-defined. Put
\begin{equation}\label{eq5}
\omega := \gdeg(\mathscr F, \Omega) - \gdeg(\scrA,B_{\varepsilon}(0)).
\end{equation}
Using $\omega$, we are going to present a result characterizing spatio-temporal symmetries of solutions to problem \eqref{eq:x-t-form}. Being of topological nature, this result allows us to completely characterize the spacial component of the symmetry in question while the temporal one can be characterized up to a folding only (in particular,  the result does {\it not} provide an information on the {\it minimal period}). To be more formal, we need
the following

\begin{definition}\label{def:extended-maximal}
  	
	(a) An orbit type $(H)$ in the space $\scrE$ is said to be of \textit{maximal kind} if there exists $k\ge 1$ and $u\neq 0, u\in \bbV_k$, such that $H=G_u$ and $(H)$ is a {\it maximal} orbit type in $\Phi (G,\bbV_k\setminus\{0\})$.
	
	\medskip
	
	(b) Take $x \in \scrE$ and assume that there exists $p \in \mathbb N$ such that $(\phi_p({G}_x)) = (H)$, where $(H)$ is of maximal kind and 
the homomorphism 	$\phi_p : O(2)\times \Gamma\times\bbZ_2\rightarrow O(2)\times \Gamma\times\bbZ_2$ is given by   
\begin{equation*}
\phi_p(g, h,\pm1)=(\mu_p(g), h,\pm1), \quad g\in O(2), \;\; h\in \Gamma
\end{equation*}
(here $\mu_p : O(2)\rightarrow O(2)/ \bbZ_p \simeq O(2)$ is the natural $p$-folding homomorphism of $O(2)$ into itself). Then, $x$ is said to have  an \textit{extended orbit type} $ (H)$.
 \end{definition}

\medskip

We are now in a position to formulate the following abstract result.
 \begin{proposition}\label{th:abstract}
 Assume that $\eta : \bfV \to \mathbb R$ satisfies ($\eta_1$)--($\eta_6$)
and let 
 $f : \bfV \times \bfV^{m-1}\times \bfV \to \bfV$ satisfy conditions (R) and  \ref{c1}\---\ref{c6} (resp.  \ref{c1}\---\ref{c5} and {\rm ($A_6'$)}).  Assume, in addition, that  $0 \not\in \sigma(\scrA)$ (cf.
\eqref{eq:5}, \eqref{eq:linearization}, \eqref{eq:linearization-formula}). Assume, finally,
\begin{equation}\label{eq:degree-actual}
       \omega=n_1(H_1)+n_2(H_2)+\dots +n_s(H_s),\quad n_j\neq 0,(H_j)\in\Phi_0(G)
\end{equation}
(cf. \eqref{eq5}). Then:
\begin{itemize}
\item[(a)] for every $j=1,2,\dots,m$, there exists a ${G}$-orbit of $2\pi$-periodic solutions $x\in \Omega$ to  \eqref{eq:x-t-form} such that 
  $(G_x) \geq (H_j)$; 
\item[(b)] if $H_j$ is finite, then the solution $x$ is non-constant; 
\item[(c)] if $(H_j)$ is of maximal kind, then the solution $x$ has the extended orbit type $(H_j)$ (cf. Definition \ref{def:extended-maximal}).
\end{itemize}
\end{proposition}
\begin{proof}
(a) Without loss of generality, one can chose $\varepsilon$ so small that $B_{\varepsilon}(0) \subset \Omega$ (cf. conditions ($\eta_4$) and ($\eta_6$)).
Put $\Omega':=\Omega\setminus \Omega_\ve$. Then, 
by the additivity property of the equivariant degree, one has: 
\begin{equation}\label{eq:differe}
\gdeg(\mathscr F,\Omega')=\gdeg(\mathscr F,\Omega)-\gdeg(\scrA,B_{\varepsilon}(0)).
\end{equation}
Next, 
combining \eqref{eq5}, \eqref{eq:degree-actual} and \eqref{eq:differe} with the existence property of the equivariant degree, implies part (a).

\medskip

(b) Follows from Remark \ref{rem:non-constant-solutions}.

\medskip

(c) Follows from Definition \ref{def:extended-maximal}.   
\end{proof}

\section{Computation of $\gdeg(\scrA, B_{\varepsilon}(0))$} \label{sec:degree-computation}

Proposition \ref{th:abstract} reduces the study of problem \eqref{eq:x-t-form} to computing $\gdeg(\scrA,B_{\varepsilon}(0))$ and $\gdeg(\mathscr F,\Omega)$. In this section,
we will develop a ``workable" formula for $\gdeg(\scrA,B_{\varepsilon}(0))$.

\subsection{Spectrum of $\scrA$}
To begin with, we collect the equivariant spectral data related to $\scrA$. Since $\scrA$ is $G$-equivariant, it respects isotypic  decomposition  
\eqref{dcp}. Put $\gamma:=e^{\frac{i2\pi}{m}}$ and $\scrA_k:=\scrA|_{\mathbb V_k}$. Keeping in mind the commensurateness of delays in problem 
 \eqref{eq:x-t-form} and taking into account \eqref{eq:action-L}, one easily obtains:
\begin{equation}\label{eq:Ak}
\scrA_k =\id+\frac{1}{k^2+1}\left(\sum _{j=0}^{m-1}\gamma^{jk}A_j-\id\right), \quad k=0,1,2\dots,
\end{equation}
where $A_j$ stands for the derivative of $f$ with respect to $j$-th variable (see condition \ref{c3})
 By assumption (R),  
$A_j = A_{m - j}$ for $j = 1,...,m-1$, hence \eqref{eq:Ak} can be simplified as follows: 
\begin{equation}\label{eq:Ak1}
\scrA_k =\id+\frac{1}{k^2+1}\left(A_0+\sum _{j=1}^{r} 2\cos\frac {2\pi jk}{m}A_j-\ve_m A_r-\id\right), \quad k=0,1,2\dots,\; r=\left\lfloor\frac {m-1}{2} \right\rfloor ,
\end{equation}
where 
\begin{equation}\label{eq:epsilon-m}
\ve_m = \begin{cases} 
1\quad  \text{if $m$ is even};\\
 0 \quad  \text{otherwise}.
\end{cases}
\end{equation}

  Since   the matrices $A_j$ are $\Gamma$-equivariant, one has $\scrA_k(V_{k,l}^-)\subset V_{k,l}^-$ $(k = 0,1,2,\ldots$ and   $l=0,1,2,\dots, \mathfrak r)$.
  In particular, 
   $A_j(V_l^-)\subset V_l^-$, so put
   \[
   A_{j,l}:=A_j|_{V_l^-}, \quad l=0,1,2,\dots, \mathfrak r.
   \] 
To simplify the computations, we will assume 
that instead of \ref{c3} the following condition is satisfied:
 \begin{itemize}
\item [($A_3'$)] $A_{j,l}=\mu_j^l \id$ for $l=0,1,2,\dots, \mathfrak r$ and $j=0,1,\dots,m-1$.
   \end{itemize}
Clearly, under the condition ($A_3'$),  the matrices $A_j$  commute with each other, therefore, condition \ref{c3} follows. 
In particular, their corresponding  eigenspaces coincide: $E(\mu_j^l)=E(\mu_{j'}^l)$. This way, one 
 obtains the following description of the spectrum of $\scrA$:
 \begin{equation}\label{eq:spectru} 
 \sigma(\mathscr A)=\bigcup_{k=0}^\infty \sigma(\mathscr A_k),
 \end{equation}
 where 
   \begin{equation}\label{spcta}
    \sigma (\scrA_k)=\left\{1+\frac{1}{1+k^2}\left(\mu^l_0+\sum _{j=1}^{r}2\cos\frac {2\pi jk}m\mu^l_j-\ve_m\mu_r^l-1\right): l=0,1,\dots,\mathfrak r, \; 
     r=\left\lfloor\frac {m-1}{2} \right\rfloor\right\}.
   \end{equation}
   
\bigskip

\subsection{Reduction to basic $G$-degrees}
For any $ l=0,1,\dots,\mathfrak r$ and $ k = 0,1,...$, put  (cf. \eqref{spcta})
\begin{equation}\label{eq:eigen-kl}
\xi_{k,l}:=1+\frac{1}{1+k^2}\left(\mu_0^l +\sum _{j=1}^{r} 2\cos \frac{2\pi jk}{m}\mu^l_j-\ve_m \mu_r^l-1 \right), \; 
r=\left\lfloor\frac {m-1}{2} \right\rfloor , \; 
\end{equation}
As is well-known (cf. \eqref{eq:prop-bas-decomp}-\eqref{eq:prod-prop}), $\xi_{k,l}$ contributes $\gdeg(\scrA,B(\scrE))$ only if $\xi_{k,l} < 0$. 
Clearly (cf. \eqref{eq:eigen-kl}), $\xi_{k,l}$
is negative (i.e. $\xi_{k,l}\in \sigma_-(\scrA)$) if and only if
\begin{align}\label{ieqk}
k^2 <-\mu_0^l-\sum _{j=1}^r  2\cos \frac{2\pi jk}{m}\mu_j^l+\ve_m\mu_r^l, \quad  l=0,1,\dots,\mathfrak r, \; r=\left\lfloor\frac {m-1}{2} \right\rfloor , \; k = 0,1,...
\end{align}
By condition ($A_3'$), the $\cV_l^-$-isotypic  multiplicity of $\mu_j^l$ is independent of $j$ and is equal to 
\begin{equation}\label{eq:m-l}
m^l := \text{dim\,} E(\mu_j^l)/\text{dim\,} \cV_l^-=\text{dim\,}V_{l}^- /\text{dim\,} \cV_l^-.
\end{equation}
Put (cf. \eqref{ieqk}-\eqref{eq:m-l})
\begin{align}\label{mtpl}   
m_{k,l}:=\begin{cases} 
m^l &\; \text {if } \;   k^2 <-\mu_0^l-\sum _{j=1}^{r}  2\cos \frac{2\pi jk}{m}\mu_j^l+\ve_m\mu_r^l  \\
0 & \text{ otherwise}.
\end{cases}
\end{align}
Then,   
\begin{align}\label{eq6}
\gdeg(\scrA,B({\scrE}))&=\prod_{k=0}^\infty 
\prod_{l=0}^{\mathfrak r}\big(\mbox{deg}_{\cV_{k,l}^-}\big)^{m_{k,l}}
\end{align}

\begin{remark}\label{rem:odd-powers-only}  
(a) Notice that in the product \eqref{eq6},  one has $m_{k,l}\not=0$ for finitely many values of $k$ and $l$ (cf. \eqref{mtpl}). Hence,  
for almost all the factors in \eqref{eq6}, one has $(\mbox{deg}_{\cV_{k,l}})^0=({G})$, which is the unit element in $A({G})$. Thus,  formula   \eqref{eq6}  is well-defined. 

\medskip 
(b) Using the relation $(\mbox{deg}_{\cV_{k,l}^-})^2=(G)$, one can further simplify formula \eqref{eq6}.  Clearly, only the exponents $m_{k,l}\not=0$ which are odd will contribute to the value of \eqref{eq6}. 
\end{remark}

\subsection{Maximal orbit types in  products of basic $G$-degrees}
In order to effectively apply Proposition \ref{th:abstract}(c), one should answer the following question: which orbit types of maximal kind (see Definition \ref{def:extended-maximal}) appearing
in the right-hand side of formula \eqref{eq6} will ``survive" in the resulting product? This question has been studied in detail in 
 \cite{ZFGW}. Here we will present one result from \cite{ZFGW} essentially used in what follows. 

To begin with, take $\deg_{\cV_{k,l}^-}$ appearing in \eqref{eq6} and let $(H_o)$ be a maximal  orbit type in $\cV_{k,l}^-\setminus \{0\}$. Then (see 
\eqref{eq:RF-0}),  
\begin{equation}\label{eq:basic}
\deg_{\cV_{k,l}^-}=(G)- x_o(H_o) + a,\quad     -x_o:=\frac{(-1)^{\text{dim}\cV_{k,l}^{-H_o}}-1}{|W(H_o)|},
\end{equation}
where $a \in A(G)$ has  a zero coefficient corresponding to $(H_o)$.  
Then, by  \eqref{eq:basic}, one has  
\begin{equation}\label{eq:coef-x-o-ireduc}
x_o=\begin{cases}
0 & \; \text{if $\text{dim}\cV_{k,l}^{-H_o}$ is even}\\ 
1 & \text{ if $\text{dim}\cV_{k,l}^{-H_o}$ is odd  and $|W(H_o)|=2$}\\
2 &    \text{ if $\text{dim}\cV_{k,l}^{-H_o}$ is odd  and $|W(H_o)|=1$}.
\end{cases}
\end{equation}
 
We need additional notations. 
\begin{definition}\label{def:coef-H0} 
   (i) For any $(H_o)\in \Phi_0(G)$, define the function $\text{coeff}^{H_o} : A(G)\to \bz$ assigning  to any  $a=\sum_{(H)} n_H(H) \in A(G)$  the coefficient  $n_{H_o}$  standing by $(H_o)$. 

\smallskip

(ii) Given an orbit  type $(H_o) \in \Phi_0(G,\mathscr E)$ of maximal kind (see Definition \ref{def:extended-maximal}(a)) and $k = 0,1,2,\dots,$ 
define the integer 
\begin{equation}\label{eq:parity}
\mathfrak n^{H_o}_k:=\sum_{l=0}^{\mathfrak r} \mathfrak l^{H_o}_{k,l} \cdot m_{k,l}, 
\end{equation}
where  $m_{k,l}$ is given by \eqref{mtpl}   and
\begin{equation}
\mathfrak l^{H_o}_{k,l}  :=\begin{cases}
1 &\text{ if } \;  \text{dim}\cV_{k,l}^{-H_o} \;\; \text{is odd}\\ 
0 &\text{ otherwise}
\end{cases} 
\end{equation}
(cf. formulas \eqref{eq6}--\eqref{eq:coef-x-o-ireduc}).
\end{definition}    
The following statement was proved in \cite{ZFGW}. 
\begin{lemma}\label{lem:comp-1} Let $(H_o) \in \Phi_0(G,\mathscr E)$ be an orbit type of
maximal kind  (see Definition \ref{def:extended-maximal}(a)) and assume that for some $k\ge 0$, the number $\mathfrak n^{H_o}_k$ is odd 
(see Definition \ref{def:coef-H0}). Then,
\begin{equation}\label{eq:formula1}
\text{\rm coeff\,}^{H_o}\Big(   \gdeg(\scrA,B_{\varepsilon}(0))  \Big)=\pm x_o,
\end{equation}
where $x_o$ is given by 
\eqref{eq:coef-x-o-ireduc}.
\end{lemma}

\section{Computation of $\gdeg(\mathscr F,\Omega)$}\label{sec:degree-computation-D}

In this section, following the scheme suggested in \cite{Amster}, where the non-equivariant case without delays was considered,
we are going to establish the following
\begin{proposition}\label{prop:comput-G-deg-FrF-Omega}
Under the assumptions  ($\eta_1$)--($\eta_6$), (R), \ref{c1}--\ref{c2} and   \ref{c4}--\ref{c6} (resp.  ($\eta_1$)--($\eta_6$), (R), \ref{c1}--\ref{c2}, 
\ref{c4}--\ref{c5} and {\rm ($A_6'$)}), one has
\begin{equation}\label{G-deg-F}
 \gdeg(\mathscr F, \Omega) =   \gdeg(\nu,C)
\end{equation}
(here $\nu : C \to S^{n-1}$ stands for the Gauss map and $O(2)$ is assumed to act trivially on $\bfV$ identified with constant $\bfV$-valued maps).
\end{proposition}  

The proof of the above proposition splits into several steps related to successive $\Omega$-admissible $G$-equivariant homotopies. 

\subsection{Outward homotopy}

To begin with,  
denote by $\mathfrak n:V\to V$ a continuous extension of the Gauss map $\nu:C\to S^{n-1}$, such that 
$|\mathfrak n(x)|\le 1$, $\mathfrak n(\gamma x) = \gamma \mathfrak n(x)$ and $\mathfrak n(-x)=-\mathfrak n(x)$ for all $x\in \bfV$ and $\gamma\in \Gamma$. Such an extension exists due to the equivariant version of the Tietze Theorem (see, for example, \cite{KB}).  

Next, for $\lambda \in [0,2]$,  define the map $f_\lambda:\bfV\times \bfV^{m-1} \times \bfV \to \bfV$ by 
\begin{equation}\label{eq:f-lambda}
f_\lambda(x,\bold y, z):= f(x,\bold y, z) + \lambda \Big(\max \big\{0, - \langle f(x,\bold y,z),  \mathfrak n(x)\rangle \big\} + {1 \over 2} \min\big\{1,\phi(z)\big\} \Big) \mathfrak n(x),
\end{equation}
where $x\in \bfV$, $\bold y\in \bfV^{m-1}$, $z\in \bfV$. 
One has the following

\begin{lemma}\label{lem:A-lambda}
Under the assumptions of Proposition \ref{prop:comput-G-deg-FrF-Omega}, the map $f_{\lambda}$ given by \eqref{eq:f-lambda} satisfies the following properties:
\begin{itemize}
\item[(R$^{\lambda}$)] $f_\lambda(x,y^1, \cdots, ,y^{m-1},z) = f_\lambda(x,y^{m-1},y^{m-2}, \cdots, y^2, y^1,z)$ 
for all  $(x,y^1,\cdots,y^{m-1},z)\in \bfV^{m+1}$,
\end{itemize}
\begin{itemize}
\item[(A$_1^{\lambda}$)]$ f_{\lambda}$ is $\Gamma$-equivariant;
\end{itemize}
\begin{itemize} 
\item[(A$_2^{\lambda}$)]
for all $x,z\in \bfV$ and $\bold y\in \bfV^{m-1}$, one has:
        	\begin{itemize}
	\item[(i)] $f_{\lambda}(x,\bold y,-z) = f_{\lambda}(z,\bold y,z)$,
	\item[(ii)] $f_{\lambda}(-x,-\bold y,z) = -f_{\lambda}(z,\bold y,z)$;
        \end{itemize}		
\end{itemize}
\begin{itemize}
\item[(A$_4^{\lambda}$)] for any $x \in C$,  $\bold y \in \bfV^{m-1}$ and $z \in \bf V$ such that $|\bold y|\le R$  and $z \perp n_x$, one has
\begin{equation}\label{eq:A4}
\langle f_{\lambda}(x,\bold y,z) ,n_x \rangle > \mathbb I_x(z);
\end{equation}
\end{itemize}
\begin{itemize}
\item[(A$_5^{\lambda}$)] for any $(x, \bold y, z) \in \bfV \times \bfV^{m-1} \times \bf V$ with  ${|x|, |\bold y| \le R}$, one has   
\[
|f_{\lambda}(x, \bold y,z)|\le 2\phi(|z|),
\]
where $\phi$ is from \ref{c5};
\end{itemize}
\begin{itemize}
\item[(A$_6^{\lambda}$)] for any $(x, \bold y, z) \in \bfV \times \bfV^{m-1} \times \bf V$ with ${|x|, |\bold y| \le R}$, 
one has 
\[
|f_{\lambda}(x,\bold y,z)|\le \nabla^2\eta(x)(z,z) 
+\langle f_{\lambda}(x,\bold y,z),\nabla \eta(x)\rangle + K+1, 
\]
provided that $f$ satisfies \ref{c6};
\end{itemize}
\begin{itemize}
\item[(${A'}_6^{\lambda}$)]
\[
\forall_{|x|\le R}\;\forall_{|\bold y|\le R}\;\forall_{z\in V}\;\;\;\; |f_{\lambda}(x,\bold y,z)|\le \alpha (\langle x, f_{\lambda}(x,\bold y, z)\rangle +|z|^2) +K ,
\]
provided that $f$ satisfies {\rm ($A_6'$)}.

\end{itemize}
\end{lemma}

\begin{proof} 

\

(R$^{\lambda}$) For any $(x,y^1,\cdots,y^{m-1}, z) \in \bfV^{m+1}$, one has: 
\begin{align*}
f_\lambda(x,y^1,\cdots, y^{m-1}, z) &= f(x,y^1,\cdots, y^{m-1}, z) \\
&+ \lambda \Big(\max \big\{0, - \langle f(x,y^1, \cdots,y^{m-1},z),  \mathfrak n(x)\rangle \big\} + {1 \over 2} \min\big\{1,\phi(z)\big\} \Big) \mathfrak n(x)\\
&= f(x,y^{m-1}, \cdots ,y^1, z) \\
&+\lambda \Big(\max \big\{0, - \langle f(x,y^{m-1},\cdots, y^{1},z),  \mathfrak n(x)\rangle \big\} + {1 \over 2} \min\big\{1,\phi(z)\big\} \Big) \mathfrak n(x)\\
&= f_\lambda(x,y^{m-1}, \cdots, y^1, z) 
\end{align*}

(A$_1^{\lambda}$) Recall that $\Gamma$ acts orthogonally on $\bfV$, $f$ and $\mathfrak n$ are  $\Gamma$-equivariant, and $\phi$ is $\Gamma$-invariant. Hence, for any $\gamma \in \Gamma$ and $(x,\bold y, z) \in \bfV^{m+1}$, one has:
\begin{align*}
f_{\lambda}(\gamma (x,\bold y,z)) &= f_{\lambda} (\gamma x, \gamma \bold y, \gamma z) \\
&= f(\gamma x,\gamma \bold y, \gamma z) + \lambda \Big(\max \big\{0, - \langle f(\gamma x, \gamma \bold y,\gamma z),  \mathfrak n(\gamma x)\rangle \big\} + {1 \over 2} \min\big\{1,\phi(\gamma z)\big\} \Big) \mathfrak n(\gamma x)\\
&= \gamma f(x,\bold y, z) + \lambda \Big(\max \big\{0, - \langle \gamma f(x, \bold y, z),  \gamma  \mathfrak n(x)\rangle \big\} + {1 \over 2} \min\big\{1,\phi(z)\big\} \Big)\gamma  \mathfrak n(x)\\
&= \gamma f_{\lambda}(x,\bold y,z).
\end{align*}

(A$_2^{\lambda}$) For any $(x,\bold y, z) \in \bfV^{m+1}$, one has (by \ref{c2}):
\begin{align*}
f_{\lambda}(x,\bold y, -z) &= f(x,\bold y, -z) +  \lambda \Big(\max \big\{0, - \langle f(x, \bold y, -z),  \mathfrak n(x)\rangle \big\} + {1 \over 2} \min\big\{1,\phi(z)\big\} \Big) \mathfrak n(x) \\
&= f(x,\bold y, z) +  \lambda \Big(\max \big\{0, - \langle f(x, \bold y, z),  \mathfrak n(x)\rangle \big\} + {1 \over 2} \min\big\{1,\phi(z)\big\} \Big) \mathfrak n(x) =
 f_{\lambda}(x,\bold y, z).
\end{align*}
Also,
\begin{align*}
f_{\lambda}(-x,-\bold y, z) &= f(-x,-\bold y, z) +  \lambda \Big(\max \big\{0, - \langle f(-x, -\bold y, z),  \mathfrak n(-x)\rangle \big\} + {1 \over 2} \min\big\{1,\phi(z)\big\} \Big) \mathfrak n(-x) \\
&= -f(x,\bold y, z) +  \lambda \Big(\max \big\{0, - \langle -f(x, \bold y, z), - \mathfrak n(x)\rangle \big\} + {1 \over 2} \min\big\{1,\phi(z)\big\} \Big) (-\mathfrak n(x))\\
&= -f_{\lambda}(x,\bold y,z).
\end{align*}

(A$_4^{\lambda}$)  For any $(x,\bold y, z) \in \bfV^{m+1}$, one has (by \ref{c4}): 
\begin{align*}
\langle f_{\lambda}(x,\bold y, z), \mathfrak n(x)\rangle &= \langle f (x,\bold y, z), \mathfrak n(x)\rangle + 
\lambda \Big(\max \big\{0, - \langle f(x, \bold y, z),  \mathfrak n(x)\rangle \big\} + {1 \over 2} \min\big\{1,\phi(z)\big\} \Big) |\mathfrak n(x)|^2 \\
& > \mathbb I_x(z) +  \lambda \Big(\max \big\{0, - \langle f(x, \bold y, z),  \mathfrak n(x)\rangle \big\} + {1 \over 2} \min\big\{1,\phi(z)\big\} \Big) \geq
\mathbb I_x(z). 
\end{align*}
Finally, to prove (A$_5^{\lambda}$), (A$_6^{\lambda}$) and (A$_7^{\lambda}$), one can use the same argument as in \cite{Amster}, p. 299.
\end{proof}  

\bigskip
Using \eqref{eq:f-lambda}, define 
the map $\mathscr F_\lambda:\mathscr E\to \mathscr E$ by 
\begin{equation}\label{eq:F-lambda}
 \scrF_\lambda(x) := x- L^{-1} \big(N_{f_\lambda}(j(x))-\mathfrak i(x)\big), \;\; x\in\scrE, 
\end{equation}
where  $N_{f_\lambda}  :  C(S^1,\bfV^{m+1}) \rightarrow  C(S^1,\bfV)$ is the  Nemytskii operator  given by 
 \begin{equation}\label{eq:Nem-lambda} 
 (N_{f_\lambda}(x,\bold y,z))(t) := f_{\lambda}(x(t),y^1(t),\dots,y^{m-1}(t),z(t)) \quad (t \in \mathbb R, \;\; \lambda \in [0,2]).
\end{equation}
\bigskip

Combining Lemma \ref{lem:A-lambda} with the definition of $\Omega$ and the argument used in the proof of Lemma \ref{lm:1}, one obtains the 
following
\begin{corollary}\label{cor:adissible-ourward}
Under the assumptions of Proposition \ref{prop:comput-G-deg-FrF-Omega}, formulas  \eqref{eq:F-lambda}--\eqref{eq:Nem-lambda} define
a $G$-equivariant $\Omega$-admissible homotopy. In particular,
\begin{equation}\label{eq:deg-F-F2}
\gdeg(\mathscr F,\Omega) = \gdeg(\mathscr F_2, \Omega).
\end{equation}
\end{corollary}
\begin{remark}\label{rem:outward} Obviously (see \eqref{eq:f-lambda} and \cite{Amster}, p. 300), the following inequality takes place: 
\begin{equation}\label{eq:outward-property}
\forall_{x\in C,\,z \in \bfV, \bold y \in \bfV^{m-1}} \qquad\langle f_2 (x, \bold y, z),\mathfrak n(x) \rangle > 0. 
\end{equation}
It follows from \eqref{eq:outward-property} that for any $x \in C$, the vector 
\begin{equation}\label{eq:Psi-for}
\Psi(x):= f_2 (x, x, \cdots,x, 0)
\end{equation}
is pointed outward  the interior of $\overline{D}$ (giving rise to the title of this subsection). Hence,  $\Psi$ and $\nu$ are $G$-equivariantly homotopic and
\begin{equation}\label{eq:deg-outward}
\gdeg(\Psi,C) = \gdeg (\nu,C)
\end{equation}

\end{remark}

\subsection{Scaling homotopy} 

To perform further deformations, we need the following
\begin{lemma}\label{lem:scaling-formula} Under the assumptions of Proposition \ref{prop:comput-G-deg-FrF-Omega},
take $M$ provided by Lemma \ref{lm:2}, $f_2$ given by \eqref{eq:f-lambda} and $\widetilde{\lambda} \in (0,1)$. Then, any 
$C^2$-smooth solution $x_{\widetilde{\lambda}} = x_{\widetilde{\lambda}}(t)$ to problem
\begin{align}\label{eq:x-t-form-lambda}
\left\{  \begin{aligned}
 \ddot{x}(t)& = \widetilde{\lambda}^2 f_2\left(x(t), \bold{x}_t, \widetilde{\lambda}^{-1} \dot x(t) \right),\quad t\in\bbR,\;x(t)\in \overline{D} \subset \mathbf{V} = \mathbb R^n,\\
 x(t)&=x(t+p),\quad\dot{x}(t) = \dot{x}(t+p)
 \end{aligned}\right.
 \end{align}
satisfies the inequality
\begin{equation}\label{eq:ap-2-lambda}
	\forall_{t\in \br}\;\;\;  |\dot x_{\widetilde{\lambda}}(t)|\le  \widetilde{\lambda} M.
\end{equation}
\end{lemma}
\begin{proof}
Put $u(t):= x_{\widetilde{\lambda}} (t/{\widetilde{\lambda}})$. Since  $\ddot{u} = \widetilde{\lambda}^{-2}\ddot{x}(t/{\widetilde{\lambda}})$, one can easily show (cf. \cite{Amster}, pp. 297-298) that 
\begin{align}\label{eq:x-t-form-lambda-1}
\left\{  \begin{aligned}
 \ddot{u}(t) &= f_2\left(u(t), u(t - \widetilde{\lambda} \tau_1), \cdots, u(t - \widetilde{\lambda} \tau_{m-1}), \dot u(t) \right),     \quad t\in\bbR,\;u(t)\in \overline{D} \subset \mathbf{V} = \mathbb R^n,\\
 u(t) &= u(t + p \widetilde{\lambda}),\quad\dot{u}(t) = \dot{u}(t+p \widetilde{\lambda})
 \end{aligned}\right.
 \end{align}
Therefore, one can use formula  \eqref{estim-M} (resp. \eqref{eq:M-from-Xiaoli}) to obtain $M_1 = M_1(2\phi, \eta, K+1, p \widetilde{\lambda},R)$ 
(resp. $M_1 = M_1(2\phi, \alpha, K+1, p \widetilde{\lambda},R)$ such that 
\begin{equation}\label{eq:est-lam-2}
\forall_{t\in \br}\;\;\;  |\dot u(t)|\le  M_1.
\end{equation} 
Since $\widetilde{\lambda} < 1$ and $\Phi$ in  \eqref{estim-M} (resp. \eqref{eq:M-from-Xiaoli}) is increasing, formula \eqref{eq:est-lam-2} combined with the 
chain rule yields \eqref{eq:ap-2-lambda}. 
\end{proof}

\medskip
\begin{remark}\label{eq:explanation-non-equiv} In contrast to problem \eqref{eq:x-t-form-lambda}, problem \eqref{eq:x-t-form-lambda-1} is not equivariant. The reader should not be confused with that: Lemma \ref{lm:2}  providing a priori bound for the first derivative of solution is {\it independent} of the symmetry conditions
\ref{c1} and \ref{c2}. 

\end{remark}

Given $u \in C(S^1,\bfV)$, denote 
\begin{equation}\label{eq:average}
\overline{u}:= {1 \over 2\pi} \int_0^{p} u(t)dt.
\end{equation}
Formula \eqref{eq:average} suggests two projections $Q_0,P_0 :   C(S^1,\bfV) \to   C(S^1,\bfV)$ given by 
\begin{equation}\label{eq:oper-Q-P}
Q_0u := \overline{u} \qquad \text{and}  \quad P_0  := \id - Q_0
\end{equation}
(as usual, we identify $\bfV$ with the image of $Q_0$ -- the subspace of constant $\bfV$-valued maps $S^1 \to \bfV$). 
Similarly to \eqref{eq:average} and \eqref{eq:oper-Q-P}, define projections $Q_2,P_2 : \scrE \to \scrE$, respectively.

For any $\widetilde{\lambda} \in (0,1)$, put 
$f_{2,\widetilde{\lambda}}(x,\bold y, z) := \widetilde{\lambda}^2 f_2\big(x, \bold{y}, \widetilde{\lambda}^{-1} z \big)$ (cf. \eqref{eq:x-t-form-lambda}) and consider a $\mu$-parameterized family of operators $\mathfrak{F}_{\widetilde{\lambda},\mu} : \scrE \to \scrE$ given by
\begin{equation}\label{eq-mu-parametr}
\mathfrak{F}_{\widetilde{\lambda},\mu}(x) := x - L^{-1}\left(Q_0N_{f_{2,\widetilde{\lambda}}} (jx) + \mu P_0N_{f_{2,\widetilde{\lambda}}} (jx)   -\mathfrak i(x)\right),
\qquad \mu \in [0,1]
\end{equation}
(here the projections $P_0$, $Q_0$ are given by \eqref{eq:average}--\eqref {eq:oper-Q-P} and $N_{f_{2,\widetilde{\lambda}}}$ denotes the corresponding Nemytskii operator).

\begin{lemma}\label{lem:crucial-lambda-mu} Under the assumptions of Proposition \ref{prop:comput-G-deg-FrF-Omega}, there exists 
$\widetilde{\lambda}_o \in (0,1]$ such that the $\mu$-parameterized family $\mathfrak{F}_{\widetilde{\lambda_o},\mu}$  (see \eqref{eq-mu-parametr}) is an $\Omega$-admissible  $G$-equivariant homotopy. In particular (cf. \eqref{eq:deg-F-F2}),
\begin{equation}\label{eq:deg-mathfrakF-2-lm}
 \gdeg(\mathfrak F,\Omega) = \gdeg(\mathfrak{F}_{\widetilde{\lambda}_o,0},\Omega),
\end{equation}
where $\mathfrak{F}_{\widetilde{\lambda}_o,0}(x) = x -   L^{-1}\left(Q_0N_{f_{2,\widetilde{\lambda}_o}} (jx)  -\mathfrak i(x)\right)$.
\end{lemma}
\begin{proof}
Following the same lines as in the proof of Lemma \ref{lem:A-lambda}((R$^{\lambda}$), (A$_1^{\lambda}$) and (A$_2^{\lambda}$)), one can easily establish that \eqref{eq-mu-parametr} is $G$-equivariant for any $\widetilde{\lambda} \in (0,1)$ and $\mu \in [0,1]$. Next, keeping in mind that 
$\mathbb I_x(\cdot)$ is a quadratic form and using (A$_4^{\lambda}$), one obtains
\begin{equation*}
 \langle f_{2,\widetilde{\lambda}}(x,\bold y,z), n_x \rangle = 
\langle \widetilde{\lambda}^2 f_{2} (x,\bold y, \widetilde{\lambda}^{-1} z), n_x\rangle >  
 \widetilde{\lambda}^2 \mathbb I_x(\widetilde{\lambda}^{-1} z, \widetilde{\lambda}^{-1} z ) =  \mathbb I_x(z),
 \end{equation*}
so that $f_{2,\widetilde{\lambda}}$ satisfies the analog of  (A$_4^{\lambda}$).
Finally, arguing by contradiction, and
combining the same idea as in \cite{Amster}, p. 300, with estimate \eqref{eq:ap-2-lambda} one arrives at the contradiction with Lemma \ref{lm:1}, from which the existence of the required $\widetilde{\lambda}_o $ follows. 
\end{proof}

\medskip

To complete the proof of Proposition \ref{prop:comput-G-deg-FrF-Omega}, it remains to establish the following
\begin{lemma}\label{lem:splitting}
Under the assumptions of Proposition \ref{prop:comput-G-deg-FrF-Omega}, one has (cf. \eqref{eq:deg-mathfrakF-2-lm})
\begin{equation}\label{lem:final-split}
\gdeg(\mathfrak{F}_{\widetilde{\lambda}_o,0},\Omega) =  \gdeg(\nu,C).
\end{equation}
\end{lemma} 
\begin{proof}
One has
\begin{align*}
\mathfrak{F}_{\widetilde{\lambda}_o,0}(x) &= x -   L^{-1}\left(Q_0N_{f_{2,\widetilde{\lambda}_o}} (jx)  -\mathfrak i(x)\right) \\
&= Q_2x + P_2x -  L^{-1}\left(Q_0N_{f_{2,\widetilde{\lambda}_o}} (jQ_2x + jP_2x)  -\mathfrak i(Q_2x + P_2x)\right)\\
&= \left(Q_2x - L^{-1}\left(Q_0N_{f_{2,\widetilde{\lambda}_o}} (jQ_2x + jP_2x)\right) + L^{-1}\mathfrak i Q_2x\right) 
+ \left( P_2x   + L^{-1}\left(\mathfrak i P_2x\right)\right)\\
&= - L^{-1}\left(Q_0N_{f_{2,\widetilde{\lambda}_o}} (jQ_2x + jP_2x)\right) +  \left( P_2x   + L^{-1}\left(\mathfrak i P_2x\right)\right)
\end{align*}
(cf. \eqref{eq:action-L}). Formula
\begin{equation}\label{eq:prelast-homot}
\mathfrak{F}_{\widetilde{\lambda}_o,0,\delta}(x) = - L^{-1}\left(Q_0N_{f_{2,\widetilde{\lambda}_o}} (jQ_2x +(1-\delta) jP_2x)\right) +  \left( P_2x   + 
(1- \delta) L^{-1}\left(\mathfrak i P_2x\right)\right), \quad \delta \in [0,1],
\end{equation}
defines a $G$-equivariant $\Omega$-admissible homotopy of $\mathfrak{F}_{\widetilde{\lambda}_o,0}$ to 
$$\widetilde{\mathfrak F}(x) := \left(- L^{-1} Q_0N_{f_{2,\widetilde{\lambda}_o}} (jQ_2x), P_2x \right)$$
(see again\eqref{eq:action-L}). Clearly, 
$$
\gdeg(\widetilde{\mathfrak F},\Omega) = \gdeg \left(- L^{-1} Q_0N_{f_{2,\widetilde{\lambda}_o}} (jQ_2), D\right) \cdot \gdeg(\id, B(P_2\mathscr E)),  
$$
where $B(P_2\mathscr E)$ stands for the unit ball in $P_2\mathscr E$. It remains to observe that 
$$\gdeg \left(- L^{-1} Q_0N_{f_{2,\widetilde{\lambda}_o}} (jQ_2), D\right) = \gdeg(\Psi,D)$$ 
(see \eqref{eq:Psi-for}) and use \eqref{eq:deg-outward}.
\end{proof}

Using the same Morse Lemma argument as in the proof of Theorem 5.6 from \cite{Amster}, one can easily establish the following
\begin{lemma}\label{eq:contractible}
Let $\eta : \bfV \to \mathbb R$ satisfy ($\eta_1$), ($\eta_4$)--($\eta_6$), and let $f : \bfV \times \bfV^{m-1} \times \bfV \to \bfV$ be a continuous map 
satisfying \ref{c5}--\ref{c6}. Then, $D$ is contractible. 
\end{lemma}

\begin{corollary}\label{cor:G-deg}
Under the assumptions of Proposition \ref{prop:comput-G-deg-FrF-Omega},
\begin{equation}\label{eq:final-F}
\gdeg(\mathfrak{F}, \Omega) = (G).
\end{equation}
\end{corollary}
\begin{proof}
Since $0 \in D$, Lemma \ref{eq:contractible} implies that  the Gauss map $\nu$ is $G$-equivariantly homotopic to the identity map and the result follows from Proposition \ref{prop:comput-G-deg-FrF-Omega}.
\end{proof}

\section{Main Results and Example}\label{sec:main}\label{sec:main-results-example}
 \subsection{Main result}
In this section, we will present our main results and describe an illustrating example with $G = O(2) \times D_8 \times \mathbb Z_2$.      The ``non-degenerate'' version of the main result is:
\begin{theorem}\label{th:main1}
Assume that  $\eta : \bf V \to \mathbb R$ satisfies ($\eta_1$)--($\eta_6$) and let  $f : \bfV \times \bfV^{m-1} \times \bfV \to \bfV$ satisfy conditions (R), \ref{c1}\---\ref{c2}, ($A_3'$), \ref{c4}--\ref{c6} (resp. (R), \ref{c1}\---\ref{c2}, ($A_3'$), \ref{c4}--\ref{c5} and {\rm ($A_6'$)}). Assume, in addition, that  $0 \not\in \sigma(\scrA)$,
where  $\sigma(\scrA)$ is given by \eqref{eq:spectru}--\eqref{spcta} (see also \eqref{eq:epsilon-m}). Assume, finally, that there  
exist  $k\in \bn$ and an orbit type $(H_o)$ in $\Phi_0(G,\mathscr E)$ of maximal kind such that  
$\mathfrak n^{H_o}_k$  is odd (see Definitions \ref{def:extended-maximal}(a) and \ref{def:coef-H0}). 

Then, system   \eqref{eq:x-t-form} admits a non-constant $2\pi$-periodic solution with the extended orbit type $(H_o)$ 
(cf. Definition \ref{def:extended-maximal}(b)).
\end{theorem}
\begin{proof}
Formulas \eqref{eq6}--\eqref{eq:coef-x-o-ireduc}
 show that  $\gdeg(\scrA,B_\varepsilon(0))   = (G) + a$, where $a$ 
has a zero coefficient corresponding to $(G)$. Hence, $\omega$ given by \eqref{eq5} has a zero coefficient corresponding to $(G)$ 
(cf. Corollary \ref{cor:G-deg}). Now, the proof follows immediately from Lemma \ref{lem:comp-1}  and Proposition \ref{th:abstract}(c).
\end{proof}

Using a similar argument, one can easily establish the following degenerate counterpart of Theorem \ref{th:main1}.     
\begin{theorem}\label{th:main2}
Assume that  $\eta : \bf V \to \mathbb R$ satisfies ($\eta_1$)--($\eta_6$) and let  $f : \bfV \times \bfV^{m-1} \times \bfV \to \bfV$ satisfy conditions (R), \ref{c1}\---\ref{c2}, ($A_3'$), \ref{c4}--\ref{c6} (resp. (R), \ref{c1}\---\ref{c2}, ($A_3'$), \ref{c4}--\ref{c5} and {\rm ($A_6'$)}). 
Put 
\[
\mathscr C:= \left\{ k\in \bn\cup\{0\}:
k^2  = -\mu_0^l-\sum _{j=1}^r  2\cos \frac{2\pi jk}{m}\mu_j^l+\ve_m\mu_r^l, \;\; l=0,1,2,\dots, \mathfrak r, \;  r:=\left\lfloor\frac {m-1}{2} \right\rfloor \right\}
\] 
and choose 
$s \in \bn$ such that 
\begin{equation}\label{eq:non-degenerate1}
\mathscr C\cap \{ (2k-1)s: k\in \bn\}=\emptyset.
\end{equation}
Assume that there  
exist  $k\in \bn$ and an orbit type $(H_o)$ in $\Phi_0(G,\mathscr E)$ of maximal kind such that  
$\mathfrak n^{H_o}_{(2k-1)s}$  is odd (see Definitions \ref{def:extended-maximal}(a) and \ref{def:coef-H0}).

Then, system   \eqref{eq:x-t-form} admits a non-constant $2\pi$-periodic solution with the extended orbit type $(H_o)$ 
(cf. Definition \ref{def:extended-maximal}(b)).

\end{theorem}

\subsection{Example} \label{sec:example}

To construct an example supporting Theorem \ref{th:main1} with condition \ref{c6} being satisfied,
take $\bfV := \mathbb R^2$ and consider the domain $D\subset \bfV$ described in polar coordinates $(r,\theta)$ as follows:
\begin{equation}\label{eq:D}
D:=\{(r,\theta) \in \br^2: 2r^4 -r^4\cos(8\theta)-1<0\}.
\end{equation}
The curve $C:= \partial D$  can be easily plotted (see Figure \ref{fig:D}). Clearly,  $D$ is invariant under the natural action  of the dihedral group $D_8=: \Gamma$ on $\bfV \simeq \mathbb C$ (in particular,
$D$ is symmetric). 
 \begin{figure}[h!]
 \centering
  \includegraphics[width=0.5\textwidth]{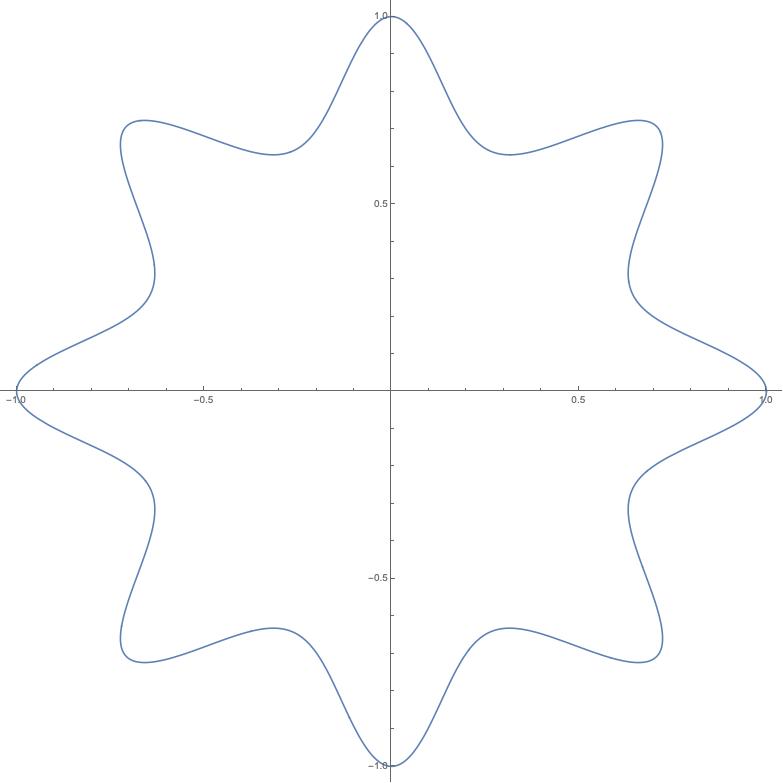}
 \caption{Domain $D$}\label{fig:D}
 \end{figure}
Since $D$ is star shaped,  
the Gauss curvature of $C$ can be easily computed as a function of $\theta$: 
\begin{equation}
\kappa(\theta)=\frac{\sqrt 2 ( -19+56 \cos(8 \theta) -3 \cos(16 \theta)) (2-\cos(8\theta))^{\frac 54}} {(13
 - 8 \cos(8 \theta)-3\cos(16\theta))^{\frac 32}}
\end{equation}
The graph of $\kappa(\theta)$ is shown on Figure \ref{fig:curvature}.
 \begin{figure}[h!]
 \centering
  \includegraphics[width=.8\textwidth]{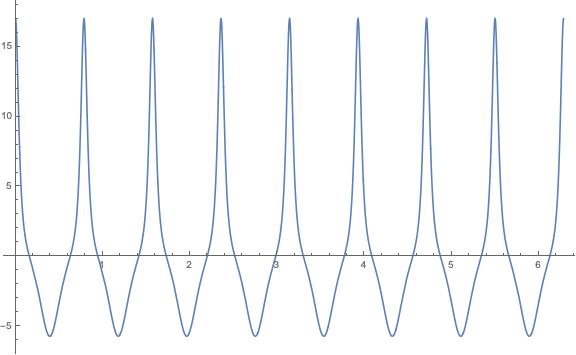}
 \caption{Curvature of $C$}\label{fig:curvature}
 \end{figure}
 
Define the function $\eta:\br^2\to \br$ given
in polar coordinates as follows:
\[
\eta(r,\theta):=2 r^4-r^4 \cos(8\theta)-1.
\]
One can easily verify (directly from the formula) that $\eta$ is $D_8$-invariant.
Passing to  Cartesian coordinates, one obtains:  
\begin{equation}\label{eq:eta-Cartesian}
\eta(x_1,x_2)=\begin{cases}
2(x_1^2+x_2^2)^{2}-\frac{x_1^8-28x_1^6x_2^2+70x_1^4x_2^4-28 x_1^2x_2^6+x_2^8}{(x_1^2+x_2^2)^{2}}-1\quad \text{if } (x_1,x_2)\not=(0,0),\\
-1 \quad\quad\quad\quad\quad\quad\quad\quad\quad\quad\quad\quad\quad\quad\quad\quad\quad\quad \quad\;\; \text{if } (x_1,x_2)\not=(0,0). 
\end{cases}
\end{equation}
By direct verification, one has:
\begin{equation*}\label{eq:etq}
\nabla\eta(x_1,x_2)=
\left[ \begin{array}{l}\frac{8 x_1(x_1^2+x_2^2)^4 - (16 x^7 - 168 x_1^5 x_2^2 + 280 x_1^3 x_2^4 - 
  56 x_1 x_2^6)(x_1^2+x_2^2) +
 4 x_1 (2 x_1^8 - 28 x_1^6 x_2^2 + 70 x_1^4 x_2^4 - 28 x_1^2 x_2^6)}{(x_1^2 + x_2^2)^{3}}\\ 
\frac{8 x_2(x_1^2+x_2^2)^4 +(56 x_1^6 x_2 - 280 x_1^4 x_2^3 + 168 x_1^2 x_2^5)(x_1^2 + x_2^2)+
 4x_2 (2 x_1^8 - 28 x_1^6 x_2^2 + 70 x_1^4 x_2^4 - 28 x_1^2 x_2^6)}{(x_1^2 + x_2^2)^{3}}
\end{array}\right]
\end{equation*}
for $(x_1,x_2) \not=(0,0)$ and 
\begin{equation*}
\lim_{\substack{x_1 \to 0 \\ x_2 \to 0}}\nabla\eta(x_1,x_2) = (0,0).
\end{equation*}
Notice that $\eta$  is of class $C^2$ and $\eta(x_1,x_2)=0$,  if and only if $x:=(x_1,x_2)\in C$, so $\eta$ satisfies 
conditions ($\eta_1$)--($\eta_6$) and $\nabla\eta(0,0)=0$. 
We are now in a position to define the required map $f : \bfV \times \bfV^{m-1} \times \bfV \to \bfV$ by  the formula
 \begin{equation}\label{eq:f-example}
 f(x,y^1,y^2,\dots, y^{m-1},z):= ( |z|^2+1) \nabla \eta(x) + \mu_0 x+\sum_{j=1}^{m-1}\mu_j y^j \quad (x, y^j, z\in \bfV),
 \end{equation}
 where $ \mu_0$ and $\mu_j$ are some constants. So far, $f$ satisfies \ref{c1}--\ref{c2} and (A$_3^{\prime}$) while constants  $ \mu_0$ and $\mu_j$ are a subject to satisfy the remaining conditions of Theorem \ref{th:main1}. 
 
To satisfy ($R$), we will assume $\mu_j = \mu_{m-j}$ for $j = 1,...,m-1$. Next, to satisfy \ref{c4}, we need to estimate $|\nabla  \eta(x(\theta))|$.
For this purpose, we will use again the polar coordinates and, by substituting $r=\root 4\of{\frac 1{2 - \cos(8 \theta)}}$,  one obtains: 
\begin{align*}
|\nabla \eta(x(\theta))|&=2 \frac{\sqrt{52 - 51 \cos(8 \theta) +4 \cos(16 \theta)-\cos(24\theta)}}{(2 - \cos{8 \theta})}\end{align*}
Observe also that 
\begin{equation}\label{eq:fund-our-case}
\mathbb I_x(z) = -\kappa(x) |z|^2,
\end{equation} 
where
\[
\kappa(x(\theta))=\frac{-\sqrt 2 (19 - 56 \cos(8 \theta) + 3 \cos(16 \theta)(2 - \cos(8 \theta))^{\frac 54}}{(
 13 - 8 \cos(8 \theta) - 3 \cos(16 \theta))^{\frac 32}},
\]
and the following estimates take place: 
\begin{equation}\label{eq:estim-Fund-nabl-curv}
17\ge \kappa(x) > -5.8, \quad 4\le   |\nabla\eta(x)| \le 21,  \quad (x \in C).
\end{equation}
We make the following assumption for \eqref{eq:f-example}:
\begin{equation}\label{eq:mu-0-mu-j-c}
 \sum_{j = 0}^{m-1} |\mu_j|>-4,
\end{equation}
and put $R := 1$ (cf. \ref{c4}). Then, combining \eqref{eq:f-example}--\eqref{eq:mu-0-mu-j-c}  with the inequality 
$|\nabla\eta(x)| + \kappa(x) >1 $ (see  Figure \ref{fig:nabla-x},  where the graph of $|\nabla\eta (x(\theta))|+\kappa((\theta))$, $x(\theta)\in C$, $\theta\in [0,2\pi]$, is shown), 
one obtains:
\begin{align*}
\langle f(x, \bold y, z), n_x \rangle &=
 (|z|^2+1) \langle \nabla\eta(x),n_x \rangle + \mu_0 \langle x,n_x \rangle + \sum_{j=1}^{m-1} \langle \mu_jy^j,n_x  \rangle\\
&=  |z|^2 |\nabla\eta(x)| +|\nabla\eta(x)|- |\mu_0| |x| + \sum_{j=1}^{m-1} \langle \mu_jy^j,n_x  \rangle\\
&\geq  |z|^2 \big( |\nabla\eta(x)|\big) +4- \sum_{j=0}^{m-1} |\mu_j| |y^j|\\
&\geq  |z|^2 \Big(-\kappa(x) + \big( |\nabla\eta(x)| + \kappa(x) \big) \Big) + 4-  \sum_{j=0}^{m-1} |\mu_j| \\
& > -|z|^2\kappa(x) = \mathbb I_x(z),
\end{align*}
so that condition \ref{c4} is satisfied. 
 \begin{figure}[h!]
 \centering
  \includegraphics[width=.8\textwidth]{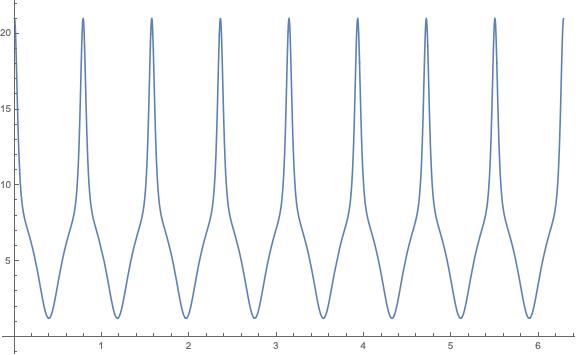}
 \caption{The values of $|\nabla\eta(x(\theta))|+\kappa(x)$ along the curve $C$. The minimal value of  $|\nabla\eta(x(\theta))|+\kappa(x)$ is larger equal than 1.22522}\label{fig:nabla-x}
 \end{figure}

It is easy to see that under the assumptions \eqref{eq:mu-0-mu-j-c},  the map \eqref{eq:f-example} satisfies condition \ref{c5} with $A:= 21$ and $B:=  \sum_{j=0}^{m-1}|\mu_j|$.

\vs

In order to show that condition {\rm ($A_6'$)} is satisfied, recall that $R=1$ and one has the following relations
for $x=(r\cos(\theta),r\sin(\theta))$, $r\le 1$  and $\alpha=4\sqrt {13}$: 
 \begin{align*}
 |f(x,\bold y,z)|&= \Big|(|z|^2+1)\nabla \eta(x)+\mu_0x+\sum_{j=1}^{m-1} \mu_jy^j\Big|\\
&\le |z|^2|\nabla\eta(x)|+21+\sum_{j=0}^{m-1} |\mu_j|\\
 &=|z^2|4 \sqrt{r^4 (4 - 4 \cos(8 \theta) + \cos^2(8 t) + 4 r^4 \sin^2(8 \theta)}+C\\
&\le \alpha\Big(\big(4 r^3 (2 - \cos(8 \theta)\big)|z|^2+|z|^2\Big) +C\\
&\le \alpha\Big(\big(|z|^2\langle x,\nabla\eta(x)\rangle +|z|^2\Big) +C\\
&\le \alpha\Big(\langle x,f(x\bold,y,z)\rangle +|z|^2\Big) +(1+\alpha)C,
 \end{align*} 
 where 
 \[
 C:=21+\sum_{j=0}^{m-1} |\mu_j|.
 \]
Clearly condition {\rm ($A_6'$)} is satisfied with $\displaystyle K:= (1+\alpha)\left(21+\sum_{j=0}^{m-1}|\mu_j|\right).$

 \medskip
 We are now in a position to apply the main Theorem \ref{th:main1} with the group $G:=O(2)\times D_8\times \bz_2$ and $V:=\br^2$ being the natural $D_8$-representation. To this end, we need to study spectrum of the linearization at the origin (see \eqref{eq:spectru}-\eqref{spcta}). We make the 
 following assumption (cf. \eqref{eq:epsilon-m}):
  \begin{equation}\label{eq:mu-cond}
\mu_0+\sum_{j=1}^r2\cos\frac{2\pi j}{m}\mu_j -\ve_m\mu_{\frac m2}<-1,
 \end{equation}
 where $r=\left\lfloor \frac{m-1}2\right\rfloor$. Then, $0 \not\in \sigma(\mathcal A)$ and 
 \[
 \sigma_-(\mathscr A):=\{\xi_0, \xi_1\}, 
 \]
 where 
 \[
 \xi_0=\mu_0+\sum_{j=1}^m \mu_j, \quad \xi_1:=1+\frac 12\left(\sum_{j=1}^r 2\cos \frac{2\pi j} m \mu_j -\ve_m\mu_{\frac m2}\right).
 \]
 In this case, formulas \eqref{mtpl}--\eqref{eq6} suggest:
 \[ \gdeg(\mathscr A,B(\mathscr E))=\deg_{\cV_{0,1}^-}\cdot \deg_{\cV_{1,1}^-},\]
 where 
 \begin{align*}
 \deg_{\cV_{0,1}^-}&:=(G)+(O(2) \times \bz_2^-)-(O(2) \times  D_2^d)-(O(2) \times  \wt D_2^d)\\
 \deg_{\cV_{1,1}^-}&:=(G)+2(\amal{D_2}{}{\bz_2}{\bz_2^-}{\wt D_2^q})+2(\amal{D_2}{}{\bz_2} {\bz_2^-}{D_2^q})+(\amal{D_2}{D_1}{\bz_2} {\bz_2^-}{\bz_2^q})-(\amal{D_2}{D_1}{\bz_2} {\wt D_2^d}{\wt D_2^q})\\
 &-(\amal{D_2}{D_1}{\bz_2}  {D_2^d}{D_2^q})-2(\amal{D_8}{}{\bz_2}{ \bz_2^-}{D_8^q}). 
 \end{align*} 
 
 \medskip
\begin{remark}\label{rem:notations} (i)  For any subgroup $S  \leq D_8$, the symbol $S^q$ stands for $S \times \mathbb Z_2$.  

\smallskip
(ii) Given two subgroups $H \leq O(2)$ and $K \leq D_8^q$, we refer to Subsection \ref{A3}  (see Appendix) for the ``amalgamated notation"  $\amal{H}{Z} {L} {R}{K}$. 

\smallskip
(iii) We refer to  \cite{AED} for the explicit description of the (sub)groups $\wt D_k$, $D_k^z$, $D_k^d$, $\wt D_k^d$, and $\mathbb Z_2^-$.
\end{remark}

\medskip

 The maximal orbit types in $\cV_{1,1}^-\setminus \{0\}$ are:
 \begin{equation}\label{eq:max-orb}
 (\amal{D_2}{D_1}{\bz_2}{\wt D_2^d}{\wt D_2^q}),\quad  (\amal{D_2}{D_1}{\bz_2} {D_2^d}{D_2^q}),\quad  (\amal{D_8}{}{D_8} {\bz_2^-}{D_8^q}). 
 \end{equation}
 We summarize our considerations in the statement following below.
 \vs
 \begin{theorem}\label{th:main-example} Assume that $D$ is given by \eqref{eq:D}.
 Let $\Gamma=D_8$, $\bfV:=\br^2$ be the natural $D_8$-representation and $f : \bfV \times \bfV^{m-1} \times \bfV \to \bfV$ be given by \eqref{eq:f-example}, where the constants $\mu_0$, $\mu_1$, \dots, $\mu_{m-1}$ satisfy conditions \eqref{eq:mu-0-mu-j-c} and 

 \eqref{eq:mu-cond}. Let $(H_o)$ be one of the orbit types listed in \eqref{eq:max-orb}. Then: 
 
 (i) $(H_o)$ of maximal type (see Definition \ref{def:extended-maximal}(a));
 
 (ii)  $\mathfrak n^{H_o}_1=1$ (see Definition \ref{def:coef-H0}); 
 
 (iii) system   \eqref{eq:x-t-form} admits a non-constant $2\pi$-periodic solution $x(t)$ with the extended orbit type $(H_o)$ (see Definition \ref{def:extended-maximal}(b)).
 \end{theorem}

Actually, for our example, the equivariant invariant $\omega=(G)-\gdeg(\mathscr A,B(\mathscr E))$ can be exactly computed using the {\it Equideg} package in GAP system:
\begin{align*}
\omega&=2(\amal{D_1}{}{\bz_2}{\bz_2^-}{D_2^d})+2(\amal{D_1}{}{\bz_2}{ \bz_2^-}{\wt D_2^d})+2(D_1 \times \bz_2^-)-2(\amal{D_2}{}{\bz_2}{\bz_2^-}{\wt D_2^q})\\
&-2(\amal{D_2}{\bz_2} {\bz_2^-}{D_2^q})-(\amal{D_2}{D_1} {\bz_2} {\bz_2^-}{D_2^d})-(D_1 \times  D_2^d)-
(\amal{D_2}{D_1}{\bz_2} {\bz_2^-}{\wt D_2^d})\\
&-(D_1 \times \wt D_2^d)-(\amal{D_2}{D_1}{\bz_2}{\bz_2^-}{\bz_2^q})+(\amal{D_2}{D_1}{\bz-2} { \wt D_2^d}{D_2^q})+(\amal{D_2}{D_1} {\bz_2}{ D_2^d}{D_2^q})\\
&+2(\amal{D_8}{}{D_8}{\bz_2^-}{D_8^q})-(O(2) \times \bz_2^-)+(O(2) \times  D_2^d)+(O(2) \times \wt D_2^d)
\end{align*}
\vskip1cm

\newpage
\appendix
\section{Equivariant Brouwer Degree Background}
\label{subsec:G-degree}

\subsection{Equivariant notation.} Below $\mathcal G$ stands for a compact Lie group.
For a subgroup $H$ of $\mathcal G$, 
denote by $N(H)$ the
normalizer of $H$ in $\mathcal G$ and by $W(H)=N(H)/H$ the Weyl group of $H$.  The symbol $(H)$ stands for the conjugacy class of $H$ in $\mathcal G$. 
Put $\Phi(\mathcal G):=\{(H): H\le \mathcal G\}$.
The set $\Phi (\mathcal G)$ has a natural partial order defined by 
$(H)\leq (K)$ iff $\exists g\in \mathcal G\;\;gHg^{-1}\leq K$. 
Put $\Phi_0 (\mathcal G):= \{ (H) \in \Phi(\mathcal G) \; : \; \text{$W(H)$  is finite}\}$.

For a $\mathcal G$-space $X$ and $x\in X$, denote by
$\mathcal G_{x} :=\{g\in \mathcal G:gx=x\}$  the {\it isotropy group}  of $x$
and call $(\mathcal G_{x})$   the {\it orbit type} of $x\in X$. Put $\Phi(\mathcal G,X) := \{(H) \in \Phi_0(\mathcal G) \; : \; 
(H) = (\mathcal G_x) \; \text{for some $x \in X$}\}$ and  $\Phi_0(\mathcal G,X):= \Phi(\mathcal G,X) \cap \Phi_0(\mathcal G)$. For a subgroup $H\leq \mathcal G$, the subspace $
X^{H} :=\{x\in X:\mathcal G_{x}\geq H\}$ is called the {\it $H$-fixed-point subspace} of $X$. If $Y$ is another $\mathcal G$-space, then a continuous map $f : X \to Y$ is called {\it equivariant} if $f(gx) = gf(x)$ for each $x \in X$ and $g \in \mathcal G$. 
Let $V$ be a finite-dimensional  $\mathcal G$-representation (without loss of generality, orthogonal).
Then, $V$  decomposes into a direct sum 
\begin{equation}
V=V_{0}\oplus V_{1}\oplus \dots \oplus V_{r},  \label{eq:Giso}
\end{equation}
where each component $V_{i}$ is {\it modeled} on the
irreducible $\mathcal G$-representation $\mathcal{V}_{i}$, $i=0,1,2,\dots ,r$, that is, $V_{i}$  contains all the irreducible subrepresentations of $V$
equivalent to $\mathcal{V}_{i}$. Decomposition  \eqref{eq:Giso}  is called  $\mathcal G$\textit{-isotypic  decomposition of} $V$.
\vs 
\subsection{ Axioms of equivariant Brouwer degree.} Denote by  $\mathcal{M}^{\mathcal G}$ the set of all admissible $\mathcal G$-pairs and let $A(\mathcal G)$ stand for the Burnside ring of $\mathcal G$ (see Introduction, item (b)). The following result (cf.  \cite{AED}) can be considered as an axiomatic definition of the {\it $\mathcal G$-equivariant Brouwer degree}.

\begin{theorem}
\label{thm:GpropDeg} There exists a unique map $\mathcal G\mbox{\rm -}\deg:\mathcal{M}
^{\mathcal G}\to A(\mathcal G)$, which assigns to every admissible $\mathcal G$-pair $(f,\Omega)$ an
element $\gdeg(f,\Omega)\in A(\mathcal G)$
\begin{equation}
\label{eq:G-deg0}\mathcal G\mbox{\rm -}\deg(f,\Omega)=\sum_{(H)}%
{n_{H}(H)}= n_{H_{1}}(H_{1})+\dots+n_{H_{m}}(H_{m}),
\end{equation}
satisfying the following properties:

\begin{description}
\item[(Existence)] If $\mathcal G\mbox{\rm -}\deg(f,\Omega)\ne
0$, i.e., $n_{H_{i}}\neq0$ for some $i$ in \eqref{eq:G-deg0}, then there
exists $x\in\Omega$ such that $f(x)=0$ and $(\mathcal G_{x})\geq(H_{i})$.

\item[(Additivity)] Let $\Omega_{1}$ and $\Omega_{2}$
be two disjoint open $\mathcal G$-invariant subsets of $\Omega$ such that
$f^{-1}(0)\cap\Omega\subset\Omega_{1}\cup\Omega_{2}$. Then,
\begin{align*}
\mathcal G\mbox{\rm -}\deg(f,\Omega)=\mathcal G\mbox{\rm -}\deg(f,\Omega_{1})+\mathcal G\mbox{\rm -}\deg
(f,\Omega_{2}).
\end{align*}

\item[(Homotopy)] If $h:[0,1]\times V\to V$ is an
$\Omega$-admissible $\mathcal G$-homotopy, then
\begin{align*}
\mathcal G\mbox{\rm -}\deg(h_{t},\Omega)=\mathrm{constant}.
\end{align*}

\item[(Normalization)] Let $\Omega$ be a $G$-invariant
open bounded neighborhood of $0$ in $V$. Then,
\begin{align*}
\mathcal G\mbox{\rm -}\deg(\id,\Omega)=(\mathcal G).
\end{align*}

\item[(Product)] For any $(f_{1},\Omega
_{1}),(f_{2},\Omega_{2})\in\mathcal{M} ^{\mathcal G}$,
\begin{align*}
\mathcal G\mbox{\rm -}\deg(f_{1}\times f_{2},\Omega_{1}\times\Omega_{2})=
\mathcal G\mbox{\rm -}\deg(f_{1},\Omega_{1})\cdot \mathcal G\mbox{\rm -}\deg(f_{2},\Omega_{2}),
\end{align*}
where the multiplication `$\cdot$' is taken in the Burnside ring $A(\mathcal G )$.

\item[(Recurrence Formula)] For an admissible $\mathcal G$-pair
$(f,\Omega)$, the $\mathcal G$-degree \eqref{eq:G-deg0} can be computed using the
following Recurrence Formula:
\begin{equation}\label{eq:RF-0}
n_{H}=\frac{\deg(f^{H},\Omega^{H})- \sum_{(K)>(H)}{n_{K}\,
n(H,K)\, \left|  W(K)\right|  }}{\left|  W(H)\right|  },
\end{equation}
where $\left|  X\right|  $ stands for the number of elements in the set $X$
and $\deg(f^{H},\Omega^{H})$ is the Brouwer degree of the map $f^{H}%
:=f|_{V^{H}}$ on the set $\Omega^{H}\subset V^{H}$.
\end{description}
\end{theorem}

The $\gdeg(f,\Omega)$ is 
 called the {\it $\mathcal G$%
-equivariant  Brouwer degree of $f$ in $\Omega$}.


\vs
\paragraph{\bf Brouwer equivariant degree of linear equivariant isomorphism:} Put $B(V):=\left\{  x\in V:\left|  x\right|  <1\right\}  $. For each
irreducible $\mathcal G$-representation $\mathcal{V} _{i}$, $i=0,1,2,\dots$, define
\begin{align}\label{eq:prop-bas-decomp}
\deg_{\mathcal{V}_{i}}:=\mathcal G\mbox{\rm -}\deg(-\id,B(\mathcal{V} _{i})),
\end{align}
and call it  
the \emph{basic degree}.

Consider a $\mathcal G$-equivariant linear isomorphism $T:V\to V$ and assume that $V$
has a $\mathcal G$-isotypic  decomposition \eqref{eq:Giso}. Then, by the
Product  property,
\begin{equation}\label{eq:prod-prop}
\mathcal G\mbox{\rm -}\deg(T,B(V))=\prod_{i=0}^{r}\mathcal G\mbox{\rm -}\deg
(T_{i},B(V_{i}))= \prod_{i=0}^{r}\prod_{\mu\in\sigma_{-}(T)} \left(
\deg_{\mathcal{V} _{i}}\right)  ^{m_{i}(\mu)}%
\end{equation}
where $T_{i}=T|_{V_{i}}$ and $\sigma_{-}(T)$ denotes the real negative
spectrum of $T$, i.e., $\sigma_{-}(T)=\left\{  \mu\in\sigma(T):\mu<0\right\}
$. \vskip.3cm

Notice that the basic degrees can be effectively computed from \eqref{eq:RF-0}: 
\begin{align*}
\deg_{\mathcal{V} _{i}}=\sum_{(H)}n_{H}(H),
\end{align*}
where 
\begin{equation}
\label{eq:bdeg-nL}n_{H}=\frac{(-1)^{\dim\mathcal{V} _{i}^{H}}- \sum
_{H<K}{n_{K}\, n(H,K)\, \left|  W(K)\right|  }}{\left|  W(H)\right|  }.
\end{equation}

\subsection{Amalgamated notation.}\label{A3} 
Given two groups $G_{1}$ and
$G_{2}$, 
the well-known result of \'E. Goursat (see \cite{DKLP,Goursat}) provides the following description of a
subgroup $\mathscr H \leq G_{1}\times G_{2}$: 
there exist subgroups
$H\leq G_{1}$ and $K\leq G_{2}$, a group $L$, and two epimorphisms
$\varphi:H\rightarrow L$ and $\psi:K\rightarrow L$ such that
\begin{equation*}
\mathscr H=\{(h,k)\in H\times K:\varphi(h)=\psi(k)\}.
\end{equation*}
The widely used notation for $\mathscr H$ is 
\begin{equation}\label{eq:amalgam-projections}
\mathscr H:=H\prescript{\varphi}{}\times_{L}^{\psi}K,
\end{equation}
in which case $H\prescript{\varphi}{}\times_{L}^{\psi}K$ is called an
\textit{amalgamated} subgroup of $G_{1}\times G_{2}$.

In this paper, we are interested in describing conjugacy classes of $\mathscr H$. Therefore, to make notation \eqref{eq:amalgam-projections} simpler and
self-contained, it is enough to indicate $L$,  
$Z=\text{Ker\thinspace}(\varphi)$ and 
$R=\text{Ker\thinspace}(\psi)$. Hence, instead of  
\eqref{eq:amalgam-projections}, we use the following notation:
\begin{equation}
\mathscr H=:H{\prescript{Z}{}\times_{L}^{R}}K~ \label{eq:amalg}.
\end{equation}

\vs
\subsection{GAP script used in this paper.}

{\small
{\bf GAP CODE:} \\
{\begin{lstlisting}[language=GAP, frame=single]
	LoadPackage( "EquiDeg" );
	# generate the groups D8, Z2, and D8 x Z2
	o2 := OrthogonalGroupOverReal( 2 );
	d8 := pDihedralGroup( 8 );                    
	z2 := pCyclicGroup( 2 );
	g1 := DirectProduct( d8, z2 );
	# generate conjugacy classes of D8 x Z2 and 
	# assigned their names 
	ccsg1 := ConjugacyClassesSubgroups (g1);
	ccsg1_names := [ "Z1", "Z2", "D1tz", "D1t",  "D1z", 
	"Z1p", "Z2m", "D1", "D2", "Z4d", "D2z", "Z2p",
	"D2td", "D1tp", "D1p", "Z4", "D2t", "D2tz", "D2d",  
	"Z4p","D4tz", "D4t", "D4","D2p", "Z8", "D4td", "D4d", 
	"D2tp", "D4z", "Z8d", "D4dh", "D8", "Z8p","D8z", "D4p", 
	"D8d", "D4tp", "D8p" ];
	ListA( ccsg1, ccsg1_names, SetAbbrv );
	# generate O(2)xD8xZ2
	g := DirectProduct( o2, g1 ); 
	ccss_g := ConjugacyClassesSubgroups( g );
	# Character Table for D8xZ2
	tbl := CharacterTable ( g1 );
	Display( tbl );
	Display( ConjugacyClasses( g1 ) );
	# the representation V_1^- is tbl[11]
	deg11:=BasicDegree( Irr( g )[1,11]);
	deg01:=BasicDegree( Irr( g )[0,11]);
	# degree of A
	degA:=deg01*deg11;
	# maximal orbit types in V[1,1]
	chi11:=Irr(g)[1,11];
	maxorbit11:=MaximalOrbitTypes(chi11);
\end{lstlisting}}
}

\vs
\noi{
\large\bf Acknowledgments:} Z. Balanov was supported by the Applied Mathematic Characteristic Discipline in Xiangnan University and Furong Scholars Award Program in Hunan Province.
W. Krawcewicz was supported by the National Natural
Science Foundation of China (No.~11871171). F. Liao was supported by   the National Natural
Science Foundation of China (No: ~11701487) and the Scientific Research Project of Hunan Province Education Department (No: 19C1700).  The computations of the equivariant degree were done using {\it EquiDeg} package for GAP systems, which was created by Hao-Pin Wu.
Symbolic computations and plotting in this paper were done using {\it Mathematica}. The authors are also grateful to Santiago Camacho from {\it Wolfram Mathematica} for his guidance in usage of {\it Mathematica}.

\end{document}